\documentclass[a4paper,11pt,oneside,reqno]{amsart}
\usepackage[utf8]{inputenc}
\usepackage{amsthm}
\usepackage{amsmath}
\usepackage{bm}
\usepackage{bbm}
\usepackage{amsfonts}
\usepackage{amssymb}
\usepackage{mathtools}
\usepackage{appendix}
\usepackage{mathrsfs}
\usepackage{setspace}
\usepackage{todonotes}
\usepackage{thmtools} % solves problem of shared counters in statements and autoref
\usepackage[foot]{amsaddr}
\usepackage[pdfdisplaydoctitle,colorlinks,breaklinks,urlcolor=blue,linkcolor=blue,citecolor=blue]{hyperref} 
\usepackage[nameinlink,capitalise]{cleveref}%must always be the last
\hypersetup{pdftitle={Kraichnan for gSQG}}

\newcommand{\E}{\mathbb{E}}
\newcommand{\F}{\mathcal{F}}

\newcommand{\N}{\mathbb{N}}

\renewcommand{\P}{\mathbb{P}}

\newcommand{\R}{\mathbb{R}}

\newcommand{\cF}{\mathcal{F}}

\newcommand{\cN}{\mathcal{N}}

\newcommand{\cS}{\mathcal{S}}

\newcommand{\cU}{\mathcal{U}}

\DeclareMathOperator{\trace}{Tr}
\DeclareMathOperator{\id}{Id}

\newcommand{\dd}{\,\mathrm{d}}
\newcommand{\eps}{\varepsilon}

\newcommand{\loc}{\mathrm{loc}}

\newtheorem{theorem}{Theorem}[section]
\newtheorem{definition}[theorem]{Definition}

\newtheorem{corollary}[theorem]{Corollary}
\newtheorem{lemma}[theorem]{Lemma}
\newtheorem{proposition}[theorem]{Proposition}

\theoremstyle{remark}
\newtheorem{remark}[theorem]{Remark}

\numberwithin{equation}{section}

\renewcommand{\leq}{\leqslant}
\renewcommand{\le}{\leqslant}
\renewcommand{\geq}{\geqslant}
\renewcommand{\ge}{\geqslant}

\title[Refined uniqueness for 2D Euler and gSQG with rough Kraichnan]{Refined uniqueness results for 2D Euler and gSQG\\ with rough Kraichnan noise}

\author[M. Bagnara]{Marco Bagnara}
\address[M. Bagnara]{Department of Mathematics, Imperial College London, UK.}
\email{m.bagnara@imperial.ac.uk}
\author[L. Galeati]{Lucio Galeati}
\address[L. Galeati]{ Dipartimento di Ingegneria e Scienze dell’Informazione e Matematica, Universita degli Studi
dell’Aquila, Italy.}
\email{lucio.galeati@univaq.it}
\begin{document}
\begin{abstract}
We prove strong well-posedness results for the stochastic 2D Euler equations in vorticity form and generalized SQG equations, with $L^p$ initial data and driven by a spatially rough, incompressible transport noise of Kraichnan type.
Previous works addressed this problem with noise of spatial regularity $\alpha\in (0,1/2)$, in a setting where a rougher noise yields a stronger regularization.
We remove this limitation by allowing any $\alpha \in (0,1)$, covering the same range of parameters for which anomalous regularization effects are known to occur in passive scalars.
In particular, this covers the physically relevant case $\alpha=2/3$, associated with the Richardson-Kolmogorov scaling of energy cascade. 
\\[1ex]
\end{abstract}

\maketitle
%\tableofcontents

%%%%%%%%%%%%%%%%%%%%%%%%%%%%%%%%%%%%%%%%%%%%%

\section{Introduction}\label{sec:introduction}
In this manuscript, we consider a class of stochastic 2D active scalar models, that comprises the 2D Euler and the generalized Surface Quasi-Geostrophic equations (gSQG for short), perturbed by a rough transport noise of Kraichnan type. Specifically, we study in $\R^2$ the SPDE
\begin{equation}\label{eq:SPDEs_Strat}\begin{cases}
		\dd \theta + u\cdot\nabla \theta \dd t + \circ \dd W\cdot\nabla \theta = 0,\\
		u = -\nabla^\perp \Lambda^{-2+\beta}\theta, \qquad \theta\vert_{t=0}=\theta_0,
\end{cases}
\end{equation}
where the unknown $\theta$ is a scalar field, while $u$ and $W$ are divergence-free velocity fields;
$u$ is recovered from $\theta$ by the nonlocal operator $\nabla^\perp \Lambda^{-2+\beta}$, where $\Lambda^s=(-\Delta)^{s/2}$ and we restrict ourselves to $\beta\in [0,1]$.
The case $\beta = 0$ corresponds to (stochastic) 2D Euler equation in vorticity form, $\beta=1$ to SQG and the intermediate range $\beta\in(0,1)$ to gSQG. 
The term $\circ \dd W$ in \eqref{eq:SPDEs_Strat} denotes stochastic integration in the Stratonovich sense, associated to a centered Gaussian field $W$ which is Brownian in time and coloured in space.
Such a field is identified by its covariance, which we uniquely determined by fixing a parameter $\alpha>0$ via the formula
\begin{equation}\label{eq:noise_covariance}\begin{split}
	&\E\big[ W_t(x)\otimes W_s(y)\big] = (s\wedge t) C(x-y),\\
	&C(z)= (2\pi)^{-1} \int_{\R^2} \langle k\rangle^{-2-2\alpha}\, P^\perp_k\, e^{i k\cdot z} \dd k\quad \text{for}\quad P^\perp_k:= I - \frac{k}{|k|}\otimes \frac{k}{|k|},\quad \forall\, k\neq 0, %\in\R^2\setminus\{0\},
\end{split}\end{equation}
where $\langle k \rangle \coloneq \sqrt{1+|k|^2}$.
The exponent $\alpha$ dictates the spatial roughness of $W$, which is $\P$-a.s. almost $C^\alpha$-regular, but never better.
For later convenience, let us also define\footnote{Here we follow \cite{drivas2025anomalous}, differently from other works (cf. \cite{coghi2023existence,GGM2024}) where $Q$ denotes the covariance itself.}
\begin{equation}\label{eq:def_of_Q}
    Q(z) \coloneq C(0)-C(z).
\end{equation}

The choice of random velocity field $W$ given by \eqref{eq:noise_covariance}, sometimes referred to as a \emph{Kraichnan noise}, was first considered in \cite{Kraichnan1968} in the context of synthetic models for passive scalar turbulence (amounting to a linear SPDE).
In the \emph{rough regime} $\alpha\in (0,1)$, this model has been recently revisited in \cite{Rowan2023,GGM2024,drivas2025anomalous,rowan2025obukhov}, providing mathematically rigorous proofs and new insights of its numerous striking features.
In particular, such a velocity field $W$ induces both anomalous dissipation and anomalous regularization effects on passive scalars, due to energy cascading from lower to higher frequencies.

A fundamental theoretical question in turbulence is to understand whether more realistic, nonlinear models of fluids can display similar phenomena; a closely related one is whether turbulent fluids can possess ``self-stabilizing'' features, due to energy being anomalously dissipating at smaller scales, preventing the formation of either blow-up or bifurcations in the dynamics. 
With this motivation in mind, the stochastically perturbed active scalar model \eqref{eq:SPDEs_Strat} is a natural test study, being sufficiently idealized yet highly nonlinear and nonlocal.
The (admittedly, a bit artificial) introduction of the stochastic perturbation $W$ is meant to mimic the unresolved small scales of an intrinsically turbulent fluid and has been advocated in different ways by several authors, see for instance \cite{BrCaFl1991,Holm2015,FlandoliPappalettera2021}; in particular, Stratonovich transport noise preserves several Lagrangian and geometric features of the original deterministic PDE.

Another reason for studying the well-posedness of weak solutions to \eqref{eq:SPDEs_Strat} is the lack of a satisfactory global-in-time solution theory for its deterministic counterpart, with explicit ill-posedness and non-uniqueness results available in the literature. For gSQG, let us mention \cite{choi2024well} for ill-posedness in H\"older spaces, \cite{JeoKim2024,cordoba2025strong} for strong norm inflation and non-existence in Sobolev scales.
Vishik \cite{vishik2018a,vishik2018b} first established non-uniqueness of the forced $2$D Euler equations for vorticities in $L^\infty_t (L^1 \cap L^p \cap \dot H^{-1})$; the result has been revisited in \cite{ABCDGJK2024,castro2024} and recently extended to gSQG and SQG in \cite{castro2025unstable}.
In the unforced case, non-uniqueness by convex integration has been shown for $2$D Euler in \cite{brue2024flexibility}, in the class $C_t (L^p \cap \dot H^{-1})$ for some $p>1$; convex integration schemes have been employed to the momentum formulation of gSQG in \cite{ma2022some, zhao2024onsager}, yielding non-unique very weak solutions (in classes of negative regularity).
For a more detailed overview of the literature, we refer to \cite[Sec. 1.4]{BGM2025}.
In this sense, the possibility to restore well-posedness by considering the stochastically perturbed SPDE \eqref{eq:SPDEs_Strat} pertains to the class of \emph{regularization by noise} phenomena.

The main goal of this note is to establish global well-posedness results for \eqref{eq:SPDEs_Strat}, for relevant classes of initial data in $L^p$.
This problem has been previously analysed in \cite{coghi2023existence,JiaoLuo2025_Euler} (for $2$D Euler) and \cite{JiaoLuo2025_mSQG,BGM2025} (for gSQG); in all the aforementioned works, uniqueness statements required the restriction $\alpha \le 1/2$, with some authors considering it to be possibly structural to their approach (see \cite[Rem. 2.14]{coghi2023existence}).
This limitation is unsatisfactory for a few reasons. First, in the context of turbulence modelling, the most relevant choice of parameter is arguably $\alpha=2/3$, see \cref{rem:numerologies} for a deeper explanation.
Secondly, the main ``coercivity estimates'' used in the proof of the aforementioned works (cf.\ \eqref{eq:intro_evolution_xi}-\eqref{eq:intro_reg_estimate} below), arising from the roughness of the noise, and how it affects the evolution of negative Sobolev norms $\dot H^{-s}$, are valid in the full range $\alpha\in (0,1)$, see \cite{coghi2023existence}; such estimates are also deeply tied to the anomalous regularization results for passive scalars from \cite{GGM2024,drivas2025anomalous,rowan2025obukhov}.
Finally, the regularization effect from $W$ gets stronger as $\alpha$ gets smaller (with a ``regularity gain'' or order $1-\alpha$, see \cite{GGM2024}), hinting that the constraint $\alpha\leq 1/2$ has no particular meaning and only makes the problem mathematically easier to handle.

Our main contribution is to establish pathwise uniqueness results for \eqref{eq:SPDEs_Strat} for any $\alpha \in (0,1)$, under suitable conditions on $p$ and $\beta$, see \cref{subsec:main_results} for the precise statements.
Moreover, we refine the existing results by dropping any $L^1$-integrability assumptions and dealing in all cases with initial data in the scaling critical class $L^{p_\star}$, see \cref{rem:comparison,rem:scaling} below.
The main novel ideas in the strategy of proof are shortly illustrated in \cref{subsec:ideas_proof}.

\subsection{Main results and comments}\label{subsec:main_results}

Given $\alpha\in (0,1)$, $\beta\in [0,1]$ such that $\alpha+\beta/2\leq 1$, we define the {\em scaling critical exponent} $p_\star=p_\star(\alpha,\beta)$ associated to \eqref{eq:SPDEs_Strat} by
\begin{equation}\label{eq:scaling_critical_exponent}
    \alpha+\frac{\beta}{2}=1-\frac{1}{p_\star}, \qquad \text{equivalently} \qquad p_\star=\frac{1}{1-\alpha-\beta/2}
\end{equation}

\cref{thm:main_Euler,thm:main_gSQG} below are the main results of the paper. We refer to \cref{subsec:notations} for a detailed explanation of the shorthand notations for function spaces adopted in the statements.

\begin{theorem}[2D Euler]\label{thm:main_Euler}
	Let $\alpha \in (0,1)$, $\beta =0$, $p\in (1,\infty)$ satisfy
	\begin{equation}\label{eq:parameters_uniqueness_Euler}
		\alpha \leq 1-\frac{1}{p}, \qquad \text{equivalently} \qquad p \geq p_\star = \frac{1}{1-\alpha}.
	\end{equation}
	Then, for any $\theta_0\in L^p \cap \dot H^{-1}$, strong existence and pathwise uniqueness hold for \eqref{eq:SPDEs_Strat}, in the sense of \cref{def:solution}, in the class
	\begin{equation*}
		\theta \in L^\infty_{\omega,t} L ^p \cap L^2_{\omega, t} ( \dot H^{-1} \cap \dot H^{-\alpha} ).
	\end{equation*}
\end{theorem}

In order to provide a similar statement for gSQG, we need to introduce a relevant class of stochastic processes.

\begin{definition}
    Let $p\in (1,\infty)$. We say that a stochastic process $\varphi$ with paths in  
    %$L^\infty_{\omega, t}L^p$
    $L^\infty_{\omega, t}(L^1+L^\infty)$
    belongs to $\cU^p$ if for every $\eps>0$ there exist $q \in (p,\infty)$ and a decomposition $\varphi=\varphi^>+\varphi^<$ such that 
    \begin{align*}
        \| \varphi^{>}\|_{ L^\infty_{\omega, t}L^p} \le \eps, \qquad \| \varphi^{<}\|_{ L^\infty_{\omega, t}L^q}<\infty.
    \end{align*}
\end{definition}

\begin{theorem}[gSQG]\label{thm:main_gSQG}
	Let $\alpha \in (0,1)$, $\beta \in (0,1)$ and $p\in (1,\infty)$ satisfy
	\begin{equation}\label{eq:parameters_uniqueness}
		%p\geq p_\star= \frac{1}{1-\alpha-\beta/2},
        \alpha+\frac{\beta}{2}+\frac{1}{p}\leq 1,
        \qquad \alpha + \beta \le 1.
	\end{equation}
	Then for any $\theta_0\in L^p\cap \dot H^{-1}$, strong existence and pathwise uniqueness hold for \eqref{eq:SPDEs_Strat}, in the sense of \cref{def:solution}, in the class
	\begin{equation*}
		\theta\in L^\infty_{\omega,t} L ^p \cap L^2_{\omega, t} ( \dot H^{-1} \cap \dot H^{-\alpha} )\cap \cU^{p_\star}.
	\end{equation*}
\end{theorem}

Let us also point out that, in both \cref{thm:main_Euler,thm:main_gSQG}, the unique solutions are obtained as strong limits of vanishing viscosity schemes, cf.\ \cref{thm:strong_existence}.
A few remarks are in order.

\begin{remark}[The space $\cU^{p_\star}$]
    The first condition in \eqref{eq:parameters_uniqueness} amounts to $p\geq p_\star$, for $p_\star$ as defined in \eqref{eq:scaling_critical_exponent}. In particular, when $p>p_\star$, it is clear from the definition that $L^\infty_{\omega,t} L^p\cap \cU^{p_\star}= L^\infty_{\omega,t} L^p$; therefore condition $\theta\in \cU^{p_\star}$ is truly needed only in the critical case $\theta_0\in L^{p_\star}$.
    For $2$D Euler, assumption $\theta\in \cU^{p_\star}$ was not needed, thanks to special asymmetric cancellations of the nonlinearity in the balance of the $\dot H^{-1}$-norm, which allow to perform a weak-strong uniqueness argument. 
\end{remark}

\begin{remark}[Further classes of initial data]
    A closer look at the proofs reveals that condition $\theta\in L^2_{\omega, t} ( \dot H^{-1} \cap \dot H^{-\alpha} )\cap \cU^{p_\star}$ suffices to deduce pathwise uniqueness; in fact, the result generalizes to $\theta_0$ admitting a similar decomposition into a small $L^{p_\star}$-component and a large $L^q$-one. This is for instance the case for Lorentz spaces, namely $\theta_0\in L^{p,q}$ with $p>p_\star$ and $q\in [1,\infty]$.
\end{remark}

\begin{remark}[The presence of forcing]
    Similarly to \cite{BGM2025}, our well-posedness results still hold in the presence of a deterministic forcing $f\in L^1_t (L^p\cap \dot H^{-1})$ on the r.h.s.\ of the SPDE; we only omit it not to make the notation too burdensome.
    This should be compared to the deterministic case, where non-uniqueness scenarios à la Vishik arise for suitable chosen $f$, cf. \cite[Thm. 1.1]{castro2025unstable}.
\end{remark}

\begin{remark}[Scaling criticality]\label{rem:scaling}
    The space $L^{p_\star}$ can be regarded as critical for the SPDE due to the following (heuristic) scaling argument.
    Let $(\theta,W)$ solve \eqref{eq:SPDEs_Strat}; given $\lambda>0$, define
    \begin{equation}\label{eq:scaling}
        \tau:=\lambda^{2(1-\alpha)}t, \quad
        y:=\lambda x, \quad
        \theta^\lambda(t,x):=\lambda^{2(1-\alpha)-\beta} \theta(\tau,y), \quad
        W^\lambda(t,x):=\lambda^{-1} W(\tau,y);
    \end{equation}
    set $u^\lambda = -\nabla^\perp \Lambda^{-2+\beta}\theta$.
    A bit formally, scaling and Stratonovich calculus rules then yield
    \begin{align*}
        & \dot \theta^\lambda(t,x) = \lambda^{4(1-\alpha)-\beta} \dot \theta(\tau , y),\qquad  u^\lambda(t,x)\cdot\nabla \theta^\lambda(t,x)
        = \lambda^{4(1-\alpha)-\beta} u(\tau,y)\cdot\nabla \theta(\tau,y),\\
        & \nabla \theta^\lambda(t,x)\cdot \dot W^\lambda(t,x)
        = \lambda^{4(1-\alpha)-\beta} \nabla \theta(\tau,y)\cdot \dot W(\tau,y)
    \end{align*}
    where $\dot W=\circ \dd W$ encodes Stratonovich integration (rigorously one should perform time changes in integral form). In particular, $(\theta^\lambda,W^\lambda)$ still solve \eqref{eq:SPDEs_Strat}; the covariance associated to $W^\lambda$ is now given by
    \begin{align*}
        \E[W^\lambda_t(x)\otimes W^\lambda_s(y)]=(s\wedge t) C^\lambda(x-y), \quad C^\lambda(z):=\lambda^{-2\alpha} C(\lambda z)
    \end{align*}
    which can be described in Fourier by (cf. \eqref{eq:noise_covariance})
    \begin{align*}
        C^\lambda(z)= (2\pi)^{-1} \int_{\R^2} \frac{1}{(\lambda^2  + |k|^2)^{1+\alpha}}\, P^\perp_k\, e^{ik\cdot z} \dd k
    \end{align*}
    As $\lambda\to 0^+$, scaling \eqref{eq:scaling} corresponds to zooming in and looking at the small scales of $\theta$ (around $(t,x)=(0,0)$). Correspondingly, $W^\lambda$ asymptotically behaves as a statistically self-similar Kraichnan noise, since $\hat C^\lambda(k)\sim |k|^{-2-2\alpha} P^\perp_k$.
    Even though $C^\lambda(0)$ explodes as $\lambda\to 0^+$, performing similar computations as in e.g. \cite[App. A]{drivas2025anomalous}, it can be shown that $Q^\lambda\coloneqq C^\lambda(0)-C^\lambda$ stays bounded with
    \begin{align*}
        Q^\lambda(z)\sim |z|^{2\alpha} \bigg[P^\parallel_z + ( 1+2\alpha) P^\perp_z\bigg] \text{ as } \lambda\to 0^+;
    \end{align*}
    under this scaling, the regularizing effect of the noise (which ultimately is due to the behaviour of $Q$ around $0$, cf. \cite{zelati2023statistically,coghi2023existence,Rowan2023,drivas2025anomalous}) is stable as $\lambda\to 0^+$.
    We may therefore regard \eqref{eq:SPDEs_Strat} as scaling quasi-invariant under \eqref{eq:scaling}.
    In this perspective, critical spaces $E$ for the initial data $\theta_0$ should satisfy
    $\| \theta_0\|_E \sim \lambda^{2(1-\alpha)-\beta} \| \theta_0(\lambda\cdot)\|_E$ as $\lambda\to 0^+$;
    in $L^p$-scales, this yields $E=L^{p_\star}$.

    Note that scaling \eqref{eq:scaling} also leaves invariant the deterministic gSQG with fractional viscosity
    \begin{align*}
		\partial_t \theta + u\cdot\nabla \theta = - \Lambda^{2(1-\alpha)}\theta, \quad
		u = -\nabla^\perp \Lambda^{-2+\beta}\theta.
    \end{align*}
\end{remark}

\begin{remark}[Comparison with previous works]\label{rem:comparison}
    As mentioned, the well-posedness of \eqref{eq:SPDEs_Strat} has been previously studied in \cite{coghi2023existence,JiaoLuo2025_Euler,JiaoLuo2025_mSQG,BGM2025}.
    For $2$D Euler, \cite{coghi2023existence,JiaoLuo2025_Euler} established strong well-posedness under the assumption
    \begin{equation}\label{eq:previous_results_Euler}
        \theta_0\in L^1\cap L^p\cap \dot H^{-1}, \quad p\in (1,\infty), \quad 0<\alpha<\min\Big\{1-\frac{1}{p},\frac12\Big\};
    \end{equation}
    gSQG equations were instead analysed in \cite{JiaoLuo2025_mSQG,BGM2025}, \cite[Thm. 1.1]{BGM2025} required at least
    \begin{equation}\label{eq:previous_results_gSQG}
        \theta_0\in L^1\cap L^p\cap \dot H^{\beta-1}, \quad
        \alpha+\frac{\beta}{2}+\frac{1}{p}\leq 1,
        \quad
        0<\frac{\beta}{2}<\alpha<\frac{1}{2}.
    \end{equation}
    In both cases, one immediately sees that both the restriction $\alpha<1/2$ and the $L^1$-integrability requirement on $\theta_0$ are removed in \cref{thm:main_Euler,thm:main_gSQG}.
    Condition $\alpha+\beta\leq 1$ in \eqref{eq:parameters_uniqueness} is rather natural: as $\alpha$ gets smaller, the regularizing effect from $W$ gets stronger and so a wider range of $\beta$ should be allowed; instead the previous condition $0<\beta/2<\alpha$ from \eqref{eq:previous_results_gSQG}, for $\alpha\ll 1$, would force one to take $\beta\ll 1$ as well, contrary to what intuition would suggest.

    Condition \eqref{eq:parameters_uniqueness_Euler} is also in line with classical Yudovich theory: when $W$ gets regular (roughly corresponding to $\alpha\uparrow 1$), $\theta_0\in \dot H^{-1}\cap L^\infty$ is required, cf.\ \cite{BFM2016}.

    Let us point out however that there are regimes covered in \eqref{eq:previous_results_gSQG} which do not fall into \eqref{eq:parameters_uniqueness}, in particular when $\beta\sim 1$ and $\alpha\sim 1/2$. We conjecture it should be possible to weaken the last condition in \eqref{eq:parameters_uniqueness} to $\alpha+\beta/2<1$; we leave this question for future research. 
\end{remark}

\begin{remark}[Numerologies and phenomenologies of turbulence]\label{rem:numerologies}
    Two-dimensional turbulent fluids $v$ are expected to display only up to $1/3$-H\"older regularity, see \cite{BofEck2012}.
    On the other hand, the Kraichnan model recreates the Richardson-Kolmogorov scaling of turbulent energy cascade for $\alpha=2/3$ (cf.\ \cite{CFG2008}, p.23).
    For this choice of $\alpha$, for $2$D Euler ($\beta=0$), \cref{thm:main_Euler} provides global well-posedness with critical parameter $p_\star=3$; by Sobolev embeddings, if $\theta\in L^3\cap H^{-1}$, then $u\in C^{1/3}$.
    In particular, the overall velocity field driving Lagrangian particles in the SPDE \eqref{eq:SPDEs_Strat} is (formally) $v=u \dd t+ \circ \dd W$, where $u\in L^\infty_t W^{1,3}$ and $\dd W \in C^{-1/2-}_t C^{2/3-}$.
    In this sense, the SPDE \eqref{eq:SPDEs_Strat} can be interpreted as a simplified description of a multiscale turbulent fluid, admitting a decomposition into large, slow scales $u$ (which belong to $C^{1/3}$, but admit higher Sobolev $W^{1,3}$), and small, fast oscillating scales $\circ\dd W$ (belonging to $C^{2/3-}$ but with very poor time regularity).
\end{remark}

\subsection{Main ideas of proof}\label{subsec:ideas_proof}

Let us shortly present the key novel ideas behind \cref{thm:main_Euler,thm:main_gSQG}; we mostly focus on the pathwise uniqueness for stochastic $2$D Euler ($\beta=0$), highlighting how to overcome the restriction $\alpha<1/2$.

Let $\theta^1$, $\theta^2$ be two solutions to \eqref{eq:parameters_uniqueness_Euler} starting from the same initial datum $\theta_0$ and set $\xi \coloneqq \theta^1 - \theta^2$, $u\coloneqq u^1-u^2$; let $\cN(\theta^i)=u^i\cdot\nabla \theta^i$ denote the nonlinear terms.
Applying It\^o's formula to $\| \xi\|_{\dot H^{-1}}^2$ and taking expectation, possibly after some regularization procedure, one finds 
\begin{equation}\label{eq:intro_evolution_xi}
    \frac{\dd}{\dd t} \E \| \xi \|^2_{\dot H^{-1}}
    = - \E \langle \Lambda^{-2} \xi, \cN(\theta^1)-\cN(\theta^2) \rangle + \E \langle \trace [ Q D^2G] \ast \xi, \xi \rangle
\end{equation}
where $G$ is the Green kernel associated with the Laplacian and $Q$ is given by \eqref{eq:def_of_Q}; in particular, the last term in the balance \eqref{eq:intro_evolution_xi} comes from the contribution of the noise $W$.

Starting with \cite{coghi2023existence}, it has become clear that when looking at the balance of negative Sobolev norms, the contributions coming from a rough Kraichnan noise display coercive features, see also \cite{GGM2024,bagnara2024anomalous,BGM2025,JiaoLuo2025_Euler,crippa2025zero}.
In our setting, one can show that there exist $\kappa_1,\kappa_2>0$ such that
\begin{equation} \label{eq:intro_reg_estimate}
    \E \langle \trace [ Q D^2G] \ast \xi, \xi \rangle \le - \kappa_1 \E \|\xi\|^2_{\dot H^{-\alpha}} + \kappa_2 \E \| \xi\|^2_{\dot H^{-1}},
\end{equation}
see \cref{subsec:covariance} for more details.
Combining \eqref{eq:intro_reg_estimate} with \eqref{eq:intro_evolution_xi}, in order to deduce that $\xi\equiv 0$ by a Gr\"onwall-type argument, it then suffices for instance to prove that
\begin{equation}\label{eq:intro_desired_estimate}
\big|\E \left[ \langle \Lambda^{-2} \xi, \cN(\theta^1)-\cN(\theta^2) \rangle \right] \big|  \le \frac{\kappa_1}{2} \E \|\xi\|^2_{\dot H^{-\alpha}} +C \E\| \xi\|_{\dot H^{-1}}^2
\end{equation}
for some $C>0$; the problem is thus reduced to an analytic estimate for the nonlinearity $\cN$.

For $2$D Euler, rearranging terms, integrating by parts and exploiting cancellations coming from the specific structure of the nonlinearity, one finds
\begin{equation}\label{eq:intro_old_nonlinear}
    \langle \Lambda^{-2} \xi, \cN(\theta^1)-\cN(\theta^2) \rangle
    = \langle \Lambda^{-2} \xi, u^1 \cdot \nabla \xi + u \cdot \nabla \theta^2\rangle
    = -\langle u^1 \cdot \nabla \Lambda^{-2} \xi, \xi \rangle.
\end{equation}
In order to understand if an estimate like \eqref{eq:intro_desired_estimate} is feasible, starting from \eqref{eq:intro_old_nonlinear}, we can perform a regularity counting.
Noting that $\| \xi\|_{\dot H^{-\alpha}}$ controls $\| \nabla\Lambda^{-2} \xi\|_{\dot H^{1-\alpha}}$, even if $u^1$ were smooth, the best we can hope for is to control $\| u^1\cdot\nabla \Lambda^{-2} \xi\|_{\dot H^{1-\alpha}}$.
In order for the duality pairing on the r.h.s. of \eqref{eq:intro_old_nonlinear} to be appropriately bounded, one then needs at least
\begin{align*}
    (1-\alpha)+(-\alpha)\geq 0\qquad \Leftrightarrow\qquad \alpha\leq \frac12.
\end{align*}
The above argument can be made rigorous, performing duality estimates and fractional Leibniz rules, and ultimately leads to the restriction on $\alpha<1/2$ appearing in \cite{coghi2023existence,BGM2025,JiaoLuo2025_Euler,JiaoLuo2025_mSQG,JiaoLuo2025_Boussinesq}.

Our main observation is that an alternative representation of the nonlinearity allows to redistribute derivatives more optimally, improving the effectiveness of the fractional Leibniz rule; this yields \eqref{eq:intro_desired_estimate}, without passing through \eqref{eq:intro_old_nonlinear}.
As a consequence, we remove the artificial restriction $\alpha<1/2$ and only require condition \eqref{eq:parameters_uniqueness_Euler}, which is more structural (cf. \cref{rem:scaling}).

As explained in \cref{subsec:nonlinear_estimates}, for $2$D Euler we have the identity 
\begin{equation*}
    \cN(\theta)=u\cdot \nabla\theta = \nabla^\perp \cdot ((u\cdot \nabla) u),\quad u=-\nabla^\perp\Lambda^{-2}\theta, 
\end{equation*}
so that after some manipulations one finds
\begin{equation}\label{eq:intro_new_nonlinear}\begin{split}
    \langle \Lambda^{-2} \xi, \cN(\theta^1)-\cN(\theta^2) \rangle
    & = -\langle \nabla^\perp \Lambda^{-2}\xi, u^1\cdot\nabla u^1- u^2\cdot\nabla u^2 \rangle\\
    & = - \langle \nabla^\perp \Lambda^{-2} \xi, (\nabla^\perp \Lambda^{-2} \xi \cdot \nabla) \nabla^\perp \Lambda^{-2} \theta^1 \rangle.
\end{split}\end{equation}
Comparing \eqref{eq:intro_new_nonlinear} to the previous \eqref{eq:intro_old_nonlinear}, there is a crucial difference: the differential operator $\nabla\nabla^\perp \Lambda^{-2}$ applied to $\theta^1$ is of zero order (which allows to exploit a priori estimates on $\| \theta^1\|_{L^p}$) while the order of derivatives landing on the two $\xi$ is perfectly balanced.
Ultimately, the rewriting \eqref{eq:intro_new_nonlinear} allows to deduce \eqref{eq:intro_desired_estimate}, under the assumptions of \cref{thm:main_Euler}.

In the gSQG case, a similar strategy can be applied by rewriting the nonlinearity in its \emph{momentum formulation}, see \cref{subsec:nonlinear_estimates}.
In this case, we can still allow any $\alpha\in (0,1)$, but we need the constraint $\alpha+\beta \le 1$, which is more natural but likely not necessary (see discussion in \cref{rem:comparison}).
Let us point out that, compared to our previous \cite{BGM2025}, here the uniqueness argument is still based on the evolution of $\| \xi\|_{\dot H^{-1}}$ and not the Hamiltonian $\| \xi\|_{\dot H^{\beta/2-1}}$.

In both cases, once estimates of the form \eqref{eq:intro_desired_estimate} are in place, they also yield \emph{stability results}, which in fact can be used to construct directly strong solutions as unique limits of vanishing viscosity approximations. This step is performed in \cref{thm:strong_existence}.

\subsection{Structure of the paper}\label{subsec:structure}
We start \cref{sec:preliminaries} by introducing the notation and conventions used throughout the work and providing the rigorous notion of solution to the SPDE.
We then present the regularization properties induced by the rough Kraichnan noise in \cref{subsec:covariance} and subsequently prove the key functional estimates for the nonlinear terms in \cref{subsec:nonlinear_estimates}.
In \cref{sec:strong_existence}, we establish strong existence by showing strong convergence of vanishing viscosity approximations, while \cref{sec:uniqueness} is devoted to the proof of pathwise uniqueness. 
\cref{app:tech_lem} collects the proofs of several technical lemmas.

\section{Preliminaries}\label{sec:preliminaries}

\subsection{Notations and conventions}\label{subsec:notations}

Given $x,y \in \R^d$, we denote by $x\cdot y$ their scalar product; $|x|$ stands for the Euclidean norm, while $\langle x \rangle \coloneqq (1+|x|^2)^{1/2}$.
We use $x\otimes y$ to indicate the tensor product $x\otimes y$ between $x$ and $y$.  
When $A\in \R^{d\times d}$ is a matrix, $\trace A$ stands for its trace, $A^T$ for its transpose.
For $x\in\R^2$, we set $x^\perp\coloneqq (-x_2, x_1)^T$.

When dealing with inequalities, we write $a\lesssim b$ if there exists a constant $c>0$ such that $a \le cb$; to stress the dependence of the hidden constant on some parameter family $\lambda$, we will sometimes write $a \lesssim_\lambda b$.

We denote by $C^\infty_c(\R^d;\R^m)$ the space of smooth functions with compact support and by $\cS(\R^d;\R^m)$ the space of Schwartz functions; its dual space of tempered distributions is denoted by $\cS'(\R^d;\R^m)$.
For $p\in [1,\infty]$, we write $L^p (\R^d;\R^m)$ for the standard Lebesgue spaces.
For $k\in \N$ and $p\in[1,\infty]$, $W^{k,p}(\R^d;\R^m)$ stand for the classical Sobolev spaces.
In general, when clear from the context, we will drop from the notation both the domain and codomain of such functions, writing for instance $L^p$ instead of $L^p (\R^d;\R^m)$.
Letting $E$ denote any of the above Banach spaces, we denote by $E_\loc$ its localized version: $f\in E_\loc$ if $f \varphi\in E$ for any $\varphi\in C^\infty_c$.
We will use the bracket $\langle \cdot,\cdot\rangle$ to represent indifferently both the inner product in $L^2$ and the duality pairings between $\cS$ and $\cS'$.

Given functions (or distributions) $f,g:\R^d\to \R^d$, we denote by $\nabla f=(\partial_j f_i)_{i,j}$ the Jacobian matrix of $f$ and write $(f\cdot\nabla)g=\sum_i f_i\partial_i g$ (which coincides with the matrix-vector product $\nabla g f$). For a scalar function $\theta$, $\nabla \theta$ denotes its gradient; for $d=2$, $\nabla^\perp\coloneqq(-\partial_2, \partial_1)$ indicates the orthogonal gradient.

We indicate by $\cF(f)$, or $\hat f$, the Fourier transform of $f$, and by $\cF^{-1}$ its inverse.
We set $\Lambda = |\nabla| \coloneqq (-\Delta)^\frac12$, defined in Fourier space as $\Lambda f= \cF^{-1}(|k| \cF(f))$; similarly, for $s\in \R$, we indicate by $\Lambda^s$ its fractional powers.

For $s\in \R$, the inhomogeneous Sobolev space $H^s=H^s(\R^d, \R^m)$ is the space of tempered distributions $f$ such that $\hat f \in L^2_\loc$ and
\begin{equation*}
    \| f \|^2_{H^s} = \| (\id - \Delta)^{s/2} f \|_{L^2}^2 = \int_{\R^d} \langle k \rangle^{2s}|\widehat f (k)|^2 \dd k < \infty. 
\end{equation*} 
Similarly, homogeneous Sobolev space $\dot H^s=\dot H^s(\R^d, \R^m)$ are made of tempered distributions $f$ such that $\hat f \in L^1_\loc$ and
\begin{equation*}
    \| f \|^2_{\dot H^s} = \| \Lambda^s f\|_{L^2}^2= \int_{\R^d} |n |^{2s}|\widehat f (n)|^2 \dd n < \infty.
\end{equation*}

In this work, the time variable $t$ will always belong to the interval $[0,T]$, for an arbitrary large but finite $T>0$. We will use the subscripts $t\in [0,T]$ and $\omega\in \Omega$ to refer to functional spaces in the time and sample variables. 
For example, given a Banach space $E$, $L^q_t E$ denotes the space $L^q([0,T]; E)$ of $q$-integrable functions with values in $E$ (defined in the Bochner sense).
Similarly, for $\gamma \in (0,1)$, $C^\gamma_t E=C^\gamma([0,T];E)$ stands for the space of $\gamma$-H\"older functions with values in $E$, and $L^\infty_{\omega,t}L^p=L^\infty(\Omega\times [0,T];L^p(\R^d))$.
Given two Banach spaces $E$ and $F$, we endow their intersection $E\cap F$ with $\|\cdot\|_{E\cap F} = \|\cdot\|_E + \|\cdot \|_F$, which makes it Banach.

\subsection{SPDE in It\^o form}\label{subsec:SPDE_ito}

While the correct physical interpretation of the SPDE \eqref{eq:SPDEs_Strat} involves Stratonovich integration, in view of the Wong-Zakai principle, it is mathematically more convenient to convert it in It\^o form, by computing the It\^o-Stratonovich correction.
As the covariance $C$ defined by \eqref{eq:noise_covariance} is isotropic, there exists a constant $c_0>0$ such that $C(0)=c_0 I$, namely\ $c_0={\rm Tr}(C(0))/2$.
The It\^o formulation of \eqref{eq:SPDEs_Strat} has been already computed in the literature, see for instance \cite[Sec. 2.2]{BGM2025} or \cite[Rem. 2.10]{coghi2023existence}, and reads as
\begin{equation}\label{eq:SPDEs_Ito}\begin{cases}
		\dd \theta + u\cdot\nabla \theta \dd t +  \dd W\cdot\nabla \theta = \frac{c_0}2\Delta \theta ,\\
		u = -\nabla^\perp \Lambda^{-2+\beta}\theta.
	\end{cases}
\end{equation}
Rigorously proving the equivalence of \eqref{eq:SPDEs_Strat} and \eqref{eq:SPDEs_Ito} requires more regularity of $\theta$ or $C$ than the one available in our setting, see the discussion in \cite[Sec. 2.2]{GalLuo2023} for more details.
Therefore, as customary for inviscid SPDEs with transport noise, we will define weak solutions relying directly on the It\^o formulation \eqref{eq:SPDEs_Ito}. 

\begin{definition} \label{def:solution}
Let $\alpha \in (0,1)$, $\beta \in [0,1]$. We say that a tuple $(\Omega, \cF, \cF_t, \P, W, \theta)$ is \emph{probabilistically weak solution} to \eqref{eq:SPDEs_Strat} (or \eqref{eq:SPDEs_Ito}) associated to the initial condition $\theta_0 \in L^1_\loc$ if 
\begin{enumerate}
    \item [i)] $(\Omega, \cF, \cF_t, \P)$ is a filtered probability space satisfying the standard assumptions and $W$ is an $\F_t$-Wiener process with covariance \eqref{eq:noise_covariance};
    \item [ii)] $\theta : \Omega \times [0,T] \to L^1_\loc$ is a $\F_t$-progressively measurable process whose trajectories are $\P$-a.s. weakly continuous in the sense of distributions and such that $\P$-a.s. $\theta,\, u, \, u \theta\, \in L^1_t L^1_\loc$, where $u= -\nabla^\perp \Lambda^{-2+\beta} \theta$;
    \item[iii)] $\P$-a.s. $\theta \in L^2_t H^{-1-\alpha}$;
    \item [iv)] for every $\varphi \in C^\infty_c$, the following holds $\P$-a.s. 
    \begin{equation}\label{eq:def_of_solution}
        \langle \theta_t ,\varphi\rangle = \langle \theta_0, \varphi \rangle + \int_0^t \left[\langle u_s \theta_s , \nabla \varphi \rangle + \frac{c_0}{2}\langle \theta_s, \Delta \varphi \rangle \right] \dd s + \int_0^t \langle \theta_s \nabla \varphi, \dd W_s  \rangle\quad \forall t \in [0,T],
    \end{equation}
\end{enumerate}
We say that $\theta$ is a \emph{probabilistically strong solution} if it is adapted to $\cF^W$, the (standard augmentation of the) filtration generated by $W$. 
\end{definition}

\begin{remark}\label{rem:defn_solution}
    Under the assumptions in \cref{def:solution}, the integrals in \eqref{eq:def_of_solution} are well defined, continuous stochastic processes.
    This is immediate for deterministic integrals, by condition ii); instead thanks to iii) and \cite[Lem. 2.8]{GalLuo2025}, the process $M_t=\int_0^t \langle \theta_s \nabla \varphi, \dd W_s  \rangle$ is a continuous local martingale, with bracket
    $[M]_t\leq \int_0^t \| \theta_s \nabla\varphi\|_{H^{-1-\alpha}}^2 \dd s$.
\end{remark}

\begin{remark}\label{rem:defn_solution2}
    A practical sufficient condition to verify the spatial regularity assumptions from ii) and iii) above is to require that $\P$-a.s. $\theta \in L^2_t (L^p \cap H^{-\alpha})$ with
    \begin{equation}\label{eq:sufficient_cond_defn_solution}
        \frac{1}{r}\coloneqq \frac{\alpha+ \beta}{2} +\frac 1p \leq 1, \quad \alpha+\beta\leq 1.
    \end{equation}
    Indeed, by H\"older's inequality and Sobolev embeddings, we have the pathwise estimate
    \begin{align*}
        \int_0^T \| u_s \theta_s\|_{L^r} \dd s 
        & \le \int_0^T \| u_s\|_{L^{\frac{2}{\beta+\alpha}}} \|\theta_s \|_{L^p}  \dd s 
        \lesssim \int_0^T \| u_s\|_{H^{1-\alpha-\beta}} \|\theta_s \|_{L^p}  \dd s\\
        & \le \int_0^T \| \theta_s\|_{H^{-\alpha}} \|\theta_s \|_{L^p}  \dd s 
        \le  \|\theta\|^2_{L^2_t (L^p \cap H^{-\alpha})}.
    \end{align*}
    Note that condition \eqref{eq:sufficient_cond_defn_solution} is always implied by our main assumption \eqref{eq:parameters_uniqueness}.
\end{remark}

\subsection{Estimates for the noise covariance}\label{subsec:covariance}

As outlined in \cref{subsec:ideas_proof} and in particular balance \eqref{eq:intro_evolution_xi}, a crucial role in the pathwise uniqueness argument is played by the coercive nature of the term $\E \langle \trace [ Q D^2G] \ast \xi, \xi \rangle$; we collect here some useful facts concerning it.

Arguing as in \cite[Sec. 3]{GGM2024}, whenever $\xi$ is regular enough (say $\xi\in H^1$), one may rewrite this term via Fourier transform as
\begin{equation}\label{eq:key_identity_fourier}
    \langle \trace [ Q D^2G] \ast \xi, \xi \rangle = \int_{\R^2} F(n) |\hat\xi(n)|^2 \dd n
\end{equation}
where
\begin{equation}\label{eq:defn_F_fourier}
    F(n) \coloneq c \int_{\R^2} \langle n-k \rangle^{-2-2\alpha} |P^\perp_{n-k} n|^2 \left(|k|^{-2} - |n|^{-2} \right)\dd k, \quad P^\perp_k\coloneq \Big(I-\frac{k}{|k|}\otimes \frac{k}{|k|}\Big) \mathbbm{1}_{k\neq 0};
\end{equation}
in the above, $c$ is a dimensional constant.
Note that $F(n)$ is finite for every $n \neq 0$, since $|P^\perp_{n-k} n| = |P^\perp_{n-k} k| \le |n| \wedge |k|$.
$F$ satisfies the following fundamental estimate: there exist constants $\kappa_1$, $\kappa_2>0$ such that
\begin{equation}\label{eq:fundamental_estimate_covariance}
    F (n) \le - \kappa_1 |n|^{-2\alpha} + \kappa_2 |n|^{-2}\quad\forall\, n\in\R^d.
\end{equation}
This is a consequence of \cite[Lem. 4.3]{coghi2023existence} and the basic estimate $\langle n\rangle^{-2\alpha} \geq \delta |n|^{-2\alpha} - C_\delta |n|^{-2}$.

When $\xi$ is not regular enough, identity \eqref{eq:key_identity_fourier} is not rigorously justified anymore and one needs to argue as in \cite{coghi2023existence} by regular approximations.
Let $p_t$ denote the heat kernel, so that $\widehat p_t (n) = e^{-4 \pi^2 |n|^2t}$.
For $0<\delta<1$, we define a regular approximation $G^\delta$ of $G$ by
\begin{equation}\label{eq:defN_Gdelta}
	G^{\delta}(x)\coloneqq \int_{\delta}^{1/\delta}p_t(x) \dd t;
\end{equation}
in particular, its Fourier transform is given by
\begin{equation}\label{eq:def_varphi}
	\widehat{G}^\delta(n)=(2\pi\vert n\vert)^{-2}(e^{-4\pi^2\vert n\vert^2\delta}-e^{-4\pi^2\vert n\vert^2/\delta}) \eqqcolon (2\pi\vert n\vert)^{-2}\widehat \chi_\delta (n),
\end{equation}
where $\widehat\chi_\delta (n) \uparrow 1$ for every $n \in \R^2$ as $\delta \to 0$.
By definition \eqref{eq:defN_Gdelta}, it is easy to see that $G^\delta$ is a Schwartz function.

\begin{lemma} \label{lem:anom_reg}
	It holds
	\begin{equation}\label{eq:anom_reg_delta}
		\left| \F \left( \trace [Q D^2G^\delta] \right)(n) \right| \lesssim |n|^{-2\alpha} + |n|^{-2}
	\end{equation}
    uniformly in $\delta\in (0,1)$.
\end{lemma}

\begin{corollary} \label{cor:reg_term_convergence}
    Let $\xi \in L^2_{\omega, t} ( \dot H^{-1} \cap \dot H^{-\alpha})$; then $\P$-a.s. it holds that 
    \begin{equation}\label{eq:regularization_convergence}\begin{split}
        \lim_{\delta\downarrow 0} \int_0^t \trace [ Q D^2G^\delta] \ast \xi_s, \xi_s \rangle \dd s
        &  = \int_0^t \int_{\R^2}  F(n) | \widehat \xi_s(n)|^2 \dd n \dd s\\
        & \leq - \kappa_1 \int_0^t \| \xi_s\|_{\dot H^{-\alpha}}^2 \dd s + \kappa_2 \int_0^t \| \xi_s\|_{\dot H^{-1}}^2 \dd s
    \end{split}\end{equation}
    where $F$ is defined as in \eqref{eq:defn_F_fourier} and convergence in \eqref{eq:regularization_convergence} holds uniformly in $t\in [0,T]$.
\end{corollary}

The proofs of \cref{lem:anom_reg} and \cref{cor:reg_term_convergence} are postponed to \cref{app:tech_lem}.

\subsection{Functional estimates for the nonlinearity} \label{subsec:nonlinear_estimates}

The next main ingredient in the proofs of \cref{thm:main_Euler,thm:main_gSQG} are analytic estimates for the nonlinear terms, employing negative Sobolev norms. 
To obtain them, we first need to introduce an alternative formulation of the nonlinearity.
For $\beta\in [0,1]$, let
\begin{equation}\label{eq:shortcut_nonlinearity}
    \cN^\beta(\theta):= u\cdot \nabla \theta = -\nabla^\perp \Lambda^{-2+\beta} \theta\cdot\nabla \theta.
\end{equation}
Recalling that the vorticity formulation of $2$D Euler is obtained from the velocity one by applying $\nabla^\perp \cdot$ on both sides of the PDE, one has
\begin{equation}\label{eq:Euler_nonlin_momentum}
    \cN^0(\theta)=u \cdot \nabla \theta = \nabla^\perp \cdot ((u \cdot \nabla) u).
\end{equation}
A similar point of view can be adopted for the gSQG equations, leading to their \emph{momentum formulation}, see for instance \cite[Sec. 2.1]{BuShVi2019} and \cite[Sec. 2.1]{zhao2024onsager}.
For any $\beta\in [0,1]$, one has
\begin{equation}\label{eq:gSQG_nonlin_momentum}
    \cN^\beta(\theta)=u \cdot \nabla \theta = \nabla^\perp \cdot ((u \cdot \nabla) h - (\nabla h)^T u), \quad
    u = -\nabla^\perp \Lambda^{-2+\beta} \theta, \quad
    h \coloneqq - \nabla^\perp \Lambda^{-2} \theta.
\end{equation}
Note that \eqref{eq:gSQG_nonlin_momentum} is consistent with \eqref{eq:Euler_nonlin_momentum}:
in the Euler case, $\beta = 0$, we have $u=h$ and the cancellation 
\begin{equation*}
    \nabla^\perp \cdot ( (\nabla u)^T u ) =\frac 12  \nabla^\perp \cdot (  \nabla |u|^2 ) =0.
\end{equation*}

Given two solutions $\theta^1$, $\theta^2$ to \eqref{eq:SPDEs_Strat}, setting $\xi=\theta^1-\theta^2$, the evolution of $\E \| \xi\|_{\dot H^{-1}}^2$  (cf. \eqref{eq:intro_evolution_xi}) gives rise to the term $\langle \Lambda^{-2} \xi, \mathcal{N}^\beta(\theta^1)-\cN^\beta(\theta^2)\rangle$.
This term can be split into further pieces by applying \eqref{eq:gSQG_nonlin_momentum} and integrating by parts (see the upcoming proof of \cref{thm:strong_existence} for more details); they can be estimated separately and corresponds to the l.h.s.\ in \eqref{eq:tril1}, \eqref{eq:tril2} and \eqref{eq:tril3} below.

\begin{lemma} \label{lem:trilinear_estimates}
	Let $\alpha \in (0,1)$, $\beta \in [0,1)$, $p\in (1,\infty)$ satisfying
	\begin{equation*}
		\alpha + \beta \leq 1, \qquad \alpha + \frac{\beta}{2} \le 1 - \frac{1}{p}.
	\end{equation*}
	Let $\xi,\varphi \in \dot H^{-\alpha}\cap \dot H^{-1}$, $\theta\in L^p$, then there exists $r_1, r_2 \in [0,1)$ such that
    \begin{subequations}
    \begin{align}
		&\left| \langle \nabla^\perp \Lambda^{-2} \xi,  (\nabla^\perp \Lambda^{-2+\beta} \varphi \cdot \nabla) \nabla^\perp  \Lambda^{-2} \theta   \rangle \right| \lesssim \| \xi \|_{\dot H^{-\alpha}}^{1-r_1} \| \xi \|_{\dot H^{-1}}^{r_1} \| \varphi \|_{\dot H^{-\alpha}}^{1-{r_2}} \| \varphi \|_{\dot H^{-1}}^{r_2}\| \theta \|_{L^p}, \label{eq:tril1}\\
        &\left| \langle \nabla^\perp \Lambda^{-2} \xi,  ( \nabla \nabla^\perp  \Lambda^{-2} \theta)^T \nabla^\perp \Lambda^{-2+\beta} \varphi    \rangle \right| \lesssim \| \xi \|_{\dot H^{-\alpha}}^{1-r_1} \| \xi \|_{\dot H^{-1}}^{r_1} \| \varphi \|_{\dot H^{-\alpha}}^{1-{r_2}} \| \varphi \|_{\dot H^{-1}}^{r_2}\| \theta \|_{L^p},\label{eq:tril2}\\
        &\left| \langle \nabla^\perp \Lambda^{-2} \xi,  ( \nabla \nabla^\perp  \Lambda^{-2} \xi)^T \nabla^\perp \Lambda^{-2+\beta} \theta    \rangle \right| \lesssim \| \xi \|_{\dot H^{-\alpha}}^{2-r_1-r_2} \| \xi \|_{\dot H^{-1}}^{r_1+r_2} \| \theta \|_{L^p}.\label{eq:tril3}
	\end{align}
    \end{subequations}
    In particular, when $\alpha + \beta/2 = 1 - 1/p$ (equivalently $p=p_\star$, the latter being defined by \eqref{eq:scaling_critical_exponent}), we have $r_1=r_2=0$, while for $p>p_\star$ we have $r_1+r_2>0$.
\end{lemma}

\begin{proof}
    By assumption, $\frac{p_\star}{p}\in (0,1]$.
    Define $s_1 := \frac{p_\star}{p} (1-\alpha)$, $s_2 := \frac{p_\star}{p} (1-\alpha-\beta)$ and let $q_i$,  $i=1,2$, be defined by the relations
    \begin{equation*}
        \frac{1}{q_i}=\frac{1}{2} - \frac{s_i}{2}.
    \end{equation*}
    Notice that $s_1\in (0,1)$ and $s_2\in [0,1)$, so that $q_1 \in (2,\infty)$ and $q_2\in [2,\infty)$. By construction, we have
    \begin{equation*}
        \frac{1}{q_1}+\frac{1}{q_2}+\frac{1}{p} = 1.
    \end{equation*}
    We start estimating the l.h.s. in \eqref{eq:tril1} by H\"older's inequality, obtaining
	\begin{align*}
			\left| \langle \nabla^\perp \Lambda^{-2} \xi,  (\nabla^\perp \Lambda^{-2+\beta} \varphi \cdot \nabla) \nabla^\perp  \Lambda^{-2} \theta   \rangle \right|
			\le \| \nabla^\perp\Lambda^{-2} \xi \|_{L^{q_1}} \| \nabla^\perp \Lambda^{-2+\beta} \varphi \|_{L^{q_2}}\|  \nabla\nabla^\perp  \Lambda^{-2} \theta\|_{L^{p}}
	\end{align*}
	Thanks to Sobolev embeddings \cite[Thm. 1.38]{Bahouri2011} and interpolation estimates \cite[Prop. 1.32]{Bahouri2011}, we have
	\begin{align*}
		&\| \nabla^\perp\Lambda^{-2} \xi \|_{L^{q_1}} \lesssim \| \nabla^\perp\Lambda^{-2} \xi \|_{\dot H^{s_1}} \le \| \xi \|_{\dot H^{s_1 - 1}} \le  \| \xi \|_{\dot H^{-\alpha}}^{1-r_1} \| \xi \|_{\dot H^{-1}}^{r_1},\\
		&\| \nabla^\perp \Lambda^{-2+\beta} \varphi \|_{L^{q_2}} \lesssim \| \nabla^\perp \Lambda^{-2+\beta} \varphi \|_{\dot H^{s_2}} \le \| \varphi \|_{\dot H^{s_2-1+\beta}} \le  \| \varphi \|_{\dot H^{-\alpha}}^{1-r_2} \| \varphi \|_{\dot H^{-1}}^{r_2},
	\end{align*}
    where
    \begin{align*}
        r_1=\frac{1-\alpha-s_1}{1-\alpha}=1-\frac{p_\star}{p}, \qquad
        r_2 = \frac{1-\alpha-\beta-s_2}{1-\alpha}=\frac{1-\alpha-\beta}{1-\alpha} \left(1-\frac{p_\star}{p} \right);
    \end{align*}
    on the other hand, Calder\'on-Zygmund theory \cite[Prop. 7.5]{Bahouri2011} provides 
	\begin{equation}\label{eq:CZ_for_theta}
		\|  \nabla\nabla^\perp  \Lambda^{-2} \theta\|_{L^{p}} \lesssim  \| \theta\|_{L^{p}}.
	\end{equation}
    Note that when $p=p_\star$, we have 
    %$s_1-1 = s_2-1+\beta= - \alpha$, meaning that $r_1=r_2=0$, while when $p>p_\star$ both $r_1$ and $r_2$ are strictly positive.
    $r_1=r_2=0$, while when $p>p_\star$ we have $r_1>0$.
    Combining the above estimates yields the proof of \eqref{eq:tril1}.
    The proof of \eqref{eq:tril2} is completely analogous, so we are left with proving \eqref{eq:tril3}.
    
    Exploiting the divergence-free property of $\nabla^\perp$ and H\"older's inequality, we have
    \begin{equation}\label{eq:estim_trilinear_proof}\begin{split}
        &\left| \langle \nabla^\perp \Lambda^{-2} \xi,  ( \nabla \nabla^\perp  \Lambda^{-2} \xi)^T \nabla^\perp \Lambda^{-2+\beta} \theta    \rangle \right|\\
        & \qquad = \left| \langle \nabla^\perp \Lambda^{-2} \xi,  ( \nabla \nabla^\perp \Lambda^{-2+\beta} \theta)^T  \nabla^\perp  \Lambda^{-2} \xi    \rangle \right|\\
        & \qquad = \left| \langle \nabla^\perp \Lambda^{-2} \xi \otimes  \nabla^\perp  \Lambda^{-2} \xi,  \nabla \nabla^\perp \Lambda^{-2+\beta} \theta     \rangle \right|\\
        & \qquad =  \left| \langle \Lambda^\beta ( \nabla^\perp \Lambda^{-2} \xi \otimes  \nabla^\perp  \Lambda^{-2} \xi),  \nabla \nabla^\perp \Lambda^{-2} \theta     \rangle \right|\\
        &\qquad \le \|  \Lambda^\beta ( \nabla^\perp \Lambda^{-2} \xi \otimes  \nabla^\perp  \Lambda^{-2} \xi) \|_{L^{\frac{q_1 q_2}{q_1 + q_2}}} \| \nabla \nabla^\perp \Lambda^{-2} \theta \|_{L^p}.
    \end{split}\end{equation}
    Then, fractional Leibniz rule \cite[Thm. 1]{GraOh2014}, Sobolev embeddings and interpolation imply that
    \begin{align*}
        \|  \Lambda^\beta ( \nabla^\perp \Lambda^{-2} \xi \otimes  \nabla^\perp  \Lambda^{-2} \xi) \|_{L^{\frac{q_1 q_2}{q_1 + q_2}}} 
        &\lesssim \| \nabla^\perp \Lambda^{-2} \xi\|_{L^{q_1}} \| \nabla^\perp \Lambda^{-2+\beta} \xi \|_{L^{q_2}} \\
        &\lesssim  \| \nabla^\perp\Lambda^{-2} \xi \|_{\dot H^{s_1}} \| \nabla^\perp \Lambda^{-2+\beta} \xi \|_{\dot H^{s_2}} \\
        &\le  \| \xi \|_{\dot H^{s_1 - 1}} \| \xi \|_{\dot H^{s_2-1+\beta}} \\
        &\le \| \xi \|_{\dot H^{-\alpha}}^{2-r_1-r_2} \| \xi \|_{\dot H^{-1}}^{r_1+r_2};
    \end{align*}
    together with \eqref{eq:CZ_for_theta} and \eqref{eq:estim_trilinear_proof}, this concludes the proof of \eqref{eq:tril3}.
\end{proof}

\section{Strong existence via vanishing viscosity approximations}\label{sec:strong_existence}

Let $\alpha\in (0,1)$, $\beta\in [0,1)$ fixed. For $\nu>0$, we consider the viscous approximation of \eqref{eq:SPDEs_Ito} given by
\begin{equation}\label{eq:viscous_SPDE}
\begin{cases}
    \dd \theta^\nu + u^\nu\cdot\nabla \theta^\nu \dd t + \dd W\cdot\nabla \theta^\nu = \left(\frac{c_0}{2}+\nu\right) \Delta\theta^\nu \dd t\\
    u^\nu = -\nabla^\perp \Lambda^{-2+\beta}\theta^\nu, \quad \theta^\nu\vert_{t=0}=\theta^\nu_0
\end{cases}\end{equation}
where $\theta^\nu_0\in L^1\cap L^\infty\cap \dot H^{-1}$.

The next statement, which collects useful properties of solutions to \eqref{eq:viscous_SPDE}, can be obtained by readapting verbatim the proofs of \cite[Lem. 3.2 \& Lem. 4.7]{BGM2025}.

\begin{proposition}\label{prop:viscous_approx}
    For any $\nu>0$ and $\theta^\nu_0\in L^1\cap L^\infty\cap \dot H^{-1}$, strong existence and pathwise uniqueness holds for solutions $\theta^\nu\in L^\infty_{\omega,t}(L^1\cap L^\infty)\cap L^2_{\omega,t} H^1$ to \eqref{eq:viscous_SPDE}; moreover we have the $\P$-a.s. pathwise bound
    \begin{equation}\label{eq:viscous_apriori1}
        \sup_{t\in [0,T]} \| \theta^\nu_t\|_{L^2}^2 +2\nu\int_0^T \| \nabla\theta^\nu_t\|_{L^2}^2 \dd t \leq 2 \| \theta^\nu_0\|_{L^2}^2.
    \end{equation}
    Furthermore, if the initial datum admits the decomposition $\theta^\nu_0=\theta^{\nu,>}_0+\theta^{\nu,<}_0$ with $\theta^{\nu,\lessgtr}_0\in L^1\cap L^\infty$, then similarly one has $\theta^\nu=\theta^{\nu,>}+\theta^{\nu,<}$, with the $\P$-a.s. estimates
    \begin{equation}\label{eq:viscous_apriori2}
        \sup_{t\in [0,T]} \| \theta^{\nu,>}_t\|_{L^q} \leq \| \theta_0^{\nu,>}\|_{L^q}, \quad
        \sup_{t\in [0,T]} \| \theta^{\nu,<}_t\|_{L^q} \leq \| \theta_0^{\nu,<}\|_{L^q}
        \quad \forall\, q\in [1,\infty].
    \end{equation}
\end{proposition}

Note that $\nu$ does not play any role in estimate \eqref{eq:viscous_apriori2}; by choosing $\theta^{\nu,>}_0=\theta^\nu_0$, $\theta^{\nu,<}_0=0$, this provides $L^\infty_{\omega,t} L^q$-bounds for $\theta^\nu$.

We can construct directly strong solutions to \eqref{eq:SPDEs_Ito} by taking the vanishing viscosity limit in the approximation scheme \eqref{eq:viscous_SPDE}. For a relevant precursor in this direction, see \cite{JiaoLuo2025_Boussinesq}.

\begin{theorem}\label{thm:strong_existence}
    Let $\alpha \in (0,1)$, $\beta \in [0,1)$ and $p\in (1,\infty)$ satisfy
	\begin{equation}\label{eq:parameters_strong_existence}
        \alpha+\frac{\beta}{2}+\frac{1}{p}\leq 1,
        \qquad \alpha + \beta \le 1.
	\end{equation}
    Let $\theta_0\in L^p\cap \dot H^{-1}$; consider a family  $\{\theta^\nu_0\}_{\nu\in (0,1]}\subset L^1\cap L^\infty\cap \dot H^{-1}$ such that
    \begin{equation}\label{eq:approx_initial_data}
        \lim_{\nu\downarrow 0}\|\theta^\nu_0-\theta_0\|_{L^p\cap \dot H^{-1}}=0, \qquad
        \lim_{\nu\downarrow 0} \nu^{2-\alpha} \| \theta^\nu_0\|_{L^2}^2 = 0,
    \end{equation}
    and denote by $\theta^\nu$ the associated solutions to \eqref{eq:viscous_SPDE}. Then the family $\{\theta^\nu\}_{\nu\in (0,1]}$ is Cauchy in $L^1_\omega C_T \dot H^{-1}\cap L^2_{\omega,t} \dot H^{-\alpha}$ and its limit $\theta$ is a probabilistically strong solution to \eqref{eq:SPDEs_Ito}, such that
    \begin{equation}\label{eq:strong_existence_regularity}
        \theta\in L^\infty_{\omega,t} L ^p \cap L^2_{\omega, t} ( \dot H^{-1} \cap \dot H^{-\alpha} )\cap \cU^{p_\star}.
    \end{equation}
\end{theorem}

\begin{remark}
    For any $p\in (1,\infty)$ and any $\theta_0\in L^p\cap \dot H^{-1}$, one can always construct an approximating family $\{\theta^\nu\}_{\nu\in (0,1]}\subset L^1\cap L^\infty\cap \dot H^{-1}$ satisfying \eqref{eq:approx_initial_data}, e.g. by applying \cref{lem:approximation} with $\eps=\nu^\delta$ in \cref{app:tech_lem} for some $\delta<1-\alpha/2$.
    Notice that \eqref{eq:parameters_strong_existence}-\eqref{eq:strong_existence_regularity} imply that the weak formulation \eqref{eq:def_of_solution} of the SPDE \eqref{eq:SPDEs_Ito} is meaningful by \cref{rem:defn_solution2}.
\end{remark}

\begin{proof}
    We divide the proof into several steps.
    
    \textit{Step 1: The set-up.}
    Let $\nu>\tilde \nu$ be fixed and set $\xi:=\theta^\nu-\theta^{\tilde \nu}$.
    Thanks to the regularity and integrability of $\theta^\nu$, $\theta^{\tilde \nu}$ (cf. \cref{prop:viscous_approx}), it's easy to check that $\cN^\beta(\theta^\nu),\, \cN^\beta(\theta^{\tilde \nu})\in L^\infty_{\omega,t} \dot H^{-1}_x$ and $\xi\in L^2_{\omega,t} \dot H^{-1}$, where we employed the shorthand notation \eqref{eq:shortcut_nonlinearity} for the nonlinearity.
    We can therefore rigorously apply It\^o formula in $\dot H^{-1}$ (cf. \cite[Thm. 2.13]{RozLot2018}) to find
    \begin{align*}
        \dd \| \xi_t\|_{\dot H^{-1}}^2
         + 2\langle \cN^\beta(\theta^\nu_t) & -\cN^\beta(\theta^{\tilde\nu}_t), \Lambda^{-2} \xi_t\rangle \dd t
        + \dd M_t \\
        & = 2\langle \nu \Delta\theta^\nu_t -\tilde\nu \Delta\theta^{\tilde\nu}_t, \Lambda^{-2} \xi_t  \rangle \dd t + \langle \trace[ Q D^2 G] \ast \xi_t, \xi_t \rangle \dd t
    \end{align*}
    where $M_t=2\int_0^t \langle \Lambda^{-2} \xi_s \nabla \xi_s, \dd W_s\rangle=-2\int_0^t \langle \xi_s \nabla \Lambda^{-2} \xi_s, \dd W_s\rangle$ is a continuous martingale thanks to \cite[Lem. 2.8]{GalLuo2025};
    the term associated to $\trace[ Q D^2 G]$ arises from the It\^o-Stratonovich corrector and the It\^o second-order term, similarly to \cite{coghi2023existence,BGM2025}.
    Writing the above in integral form, rearranging terms, we have the $\P$-a.s.\ pathwise estimate
    \begin{align*}
        \| \xi_t\|_{\dot H^{-1}}^2 & -\|\xi_0\|_{\dot H^{-1}}^2  + \int_0^t 2\langle \cN^\beta(\theta^\nu_s) -\cN^\beta(\theta^{\tilde\nu}_s), \Lambda^{-2} \xi_s\rangle \dd s
        + M_t\\
        & = \int_0^t \big[ -2\tilde \nu\| \xi_s\|_{L^2}^2 \dd s +2(\nu-\tilde \nu) \langle \Delta\theta^\nu_s, \Lambda^{-2}\xi_s\rangle \dd s\big]
        + \int_0^t \langle \trace[ Q D^2 G] \ast \xi_s, \xi_s \rangle \dd s\\
        & \leq 2 \nu \int_0^t |\langle \theta^\nu_s, \xi_s\rangle| \dd s -\kappa_1 \int_0^t \| \xi_s\|_{\dot H^{-\alpha}}^2 \dd s + \kappa_2 \int_0^t \| \xi_s\|_{\dot H^{-1}}^2 \dd s
    \end{align*}
    where in the last estimate we used that $\Delta \Lambda^{-2} \xi=-\xi$ and we applied \eqref{eq:key_identity_fourier}, together with the second part of \eqref{eq:regularization_convergence}.

    \textit{Step 2: Gr\"onwall-type estimates and Cauchy property.}
    We claim that there exists a deterministic constant $C_1$, depending on $\theta_0$ and the approximation scheme $\{\theta^\nu_0\}_\nu$, but not on $\nu,\tilde\nu$, such that $\P$-a.s.
    \begin{equation}\label{eq:strong_exist_proof_eq1}
        2|\langle \cN^\beta(\theta^\nu_s) -\cN^\beta(\theta^{\tilde\nu}_s), \Lambda^{-2} \xi_s\rangle| \leq \frac{\kappa_1}{2} \|\xi_s\|_{\dot H^{-\alpha}}^2 + C_1 \| \xi_s\|_{\dot H^{-1}}^2\quad \forall\, s\in [0,T].
    \end{equation}
    We postpone the verification of \eqref{eq:strong_exist_proof_eq1} to Step 3 below and proceed with the argument. Combining the estimates from Step 1 with \eqref{eq:strong_exist_proof_eq1} yields
    \begin{equation}\label{eq:strong_exist_proof_eq2}\begin{split}
        \| \xi_t\|_{\dot H^{-1}}^2 & + \frac{\kappa_1}{2} \int_0^t \| \xi_s\|_{\dot H^{-\alpha}}^2 \dd s + M_t
        \leq  \|\xi_0\|_{\dot H^{-1}}^2 + C_2\int_0^t \|\xi_s\|_{\dot H^{-1}}^2 \dd s + \int_0^t 2\nu |\langle \theta^\nu_s, \xi_s\rangle| \dd s
    \end{split}\end{equation}
    with $C_2= C_1 + \kappa_2$. We can estimate the last integral by duality, as
    \begin{align*}
        2\nu |\langle \theta^\nu_s, \xi_s\rangle|
        \leq 2\nu \|  \theta^\nu_s\|_{\dot H^\alpha } \| \xi_s\|_{\dot H^{-\alpha}}
        \leq C_3 \nu^2 \|  \theta^\nu_s\|_{\dot H^\alpha }^2 + \frac{\kappa_1}{4} \| \xi_s\|_{\dot H^{-\alpha}}^2
    \end{align*}
    where we employed Young's inequality with some $C_3=C_3(\kappa_1)$.
    Moreover by interpolation estimates and the pathwise bound \eqref{eq:viscous_apriori1}, $\P$-a.s. it holds that
    \begin{align*}
        \nu^2 \int_0^t \|  \theta^\nu_s\|_{\dot H^\alpha }^2 \dd s
        \leq \nu^2 \bigg(\int_0^T \|  \theta^\nu_s\|_{\dot H^1 }^2 \dd s\bigg)^\alpha \bigg(\int_0^T \|  \theta^\nu_s\|_{L^2}^2 \dd s\bigg)^{1-\alpha}
        \lesssim T^{1-\alpha} \nu^{2-\alpha}  \| \theta^\nu_0\|_{L^2}^2.
    \end{align*}
    Define the process $Z_t \coloneq \| \xi_t\|_{\dot H^{-1}}^2 +  \frac{\kappa_1}{2}\int_0^t \| \xi_s\|_{\dot H^{-\alpha}}^2 \dd s$; reinserting these estimates in \eqref{eq:strong_exist_proof_eq2}, we arrive at the $\P$-a.s. pathwise bound
    \begin{equation}\label{eq:strong_exist_proof_eq3}
        Z_t + M_t \leq Z_0 + C_2 \int_0^t Z_s \dd s + C_4 T^{1-\alpha} \nu^{2-\alpha}  \| \theta^\nu_0\|_{L^2}^2\quad \forall\, t\in [0,T].
    \end{equation}
    From \eqref{eq:strong_exist_proof_eq3}, we can perform two kinds of estimates.
    First, taking expectation on both sides, using that $\E[M_t]=0$ and Fubini's theorem, one finds
    \begin{align*}
        \E[Z_t] \leq \| \theta^\nu_0-\theta^{\tilde \nu}_0\|_{\dot H^{-1}}^2 + C_2 \int_0^t \E[ Z_s] \dd s + C_4 T^{1-\alpha} \nu^{2-\alpha}  \| \theta^\nu_0\|_{L^2}^2;
    \end{align*}
    classical Gr\"onwall lemma then yields
    \begin{equation}\label{eq:Cauchy_v1}
        \sup_{t\in [0,T]} \E[\| \theta^\nu_t-\theta^{\tilde \nu}_t\|_{\dot H^{-1}}^2] + \E\Big[ \int_0^T \| \theta^\nu_t-\theta^{\tilde \nu}_t\|_{\dot H^{-\alpha}}^2 \dd t \Big]
        \lesssim_{C_i, T}  \| \theta^\nu_0-\theta^{\tilde \nu}_0\|_{\dot H^{-1}}^2 + \nu^{2-\alpha} \| \theta^\nu_0\|_{L^2}^2.
    \end{equation}
    By assumption \eqref{eq:approx_initial_data}, the r.h.s. converges to $0$ as $\nu,\tilde\nu\to 0^+$, thus showing that $\{\theta^\nu\}_\nu$ is Cauchy in $L^2_{\omega,t} \dot H^{-\alpha}$ and that, for any fixed $t\in [0,T]$, $\{\theta^\nu_t\}_\nu$ is Cauchy in $L^2_\omega \dot H^{-1}$.

    Secondly, following \cite{Sch2013}, we can apply Gr\"onwall at a pathwise level in \eqref{eq:strong_exist_proof_eq3} and integrate by parts to find
    \begin{align*}
        Z_t \leq e^{C_2 t} \Big( Z_0  + C_4 \nu^{2-\alpha} T^{1-\alpha} \| \theta^\nu_0\|_{L^2}^2\Big) + L_t
    \end{align*}
    where $L$ is another continuous martingale starting at $0$. As a consequence, for any stopping time $\tau$, it holds
    \begin{align*}
        \E[Z_\tau] \leq e^{C_2 T} \Big( \|\theta^\nu_0-\theta^{\tilde\nu}_0\|_{\dot H^{-1}}^2  + C_4 \nu^{2-\alpha} T^{1-\alpha} \| \theta^\nu_0\|_{L^2}^2\Big).
    \end{align*}
    By Lenglart's inequality  (cf. \cite{GeiSch2021}), for any $\delta\in (0,1)$, we deduce that \begin{equation}\label{eq:Cauchy_v2}
        \E\left[ \sup_{t\in [0,T]}\| \theta^\nu_t-\theta^{\tilde\nu}_t\|_{\dot H^{-1}}^{2\delta} \right] \lesssim_{\delta, C_i, T} \Big(\| \theta^\nu_0-\theta^{\tilde\nu}_0\|_{\dot H^{-1}}^2+\nu^{2-\alpha} \| \theta^\nu_0\|_{L^2}^2\Big)^\delta.
    \end{equation}
    As before, the r.h.s. in \eqref{eq:Cauchy_v2} vanishes as $\nu,\tilde\nu\to 0^+$ due to \eqref{eq:approx_initial_data}, showing that $\{\theta^\nu\}_\nu$ is a Cauchy sequence in $L^{2\delta}_\omega C_t \dot H^{-1}_x$, for any $\delta\in (0,1)$.

    \textit{Step 3: Functional estimates.}
    We now present the proof of estimate \eqref{eq:strong_exist_proof_eq1}.

   We can find $q\in (p_\star,\infty)$ with the following property: for any $\eps>0$, there exists a decomposition $\theta^\nu_0=\theta^{\nu,>}_0+\theta^{\nu,<}_0$ such that $\sup_{\nu>0} \| \theta^{\nu,>}_0\|_{L^{p_\star}}\leq \eps$ and $\sup_{\nu>0} \| \theta^{\nu,<}_0\|_{L^q}\leq C_\eps$.
    Indeed, when $p>p_\star$, we can trivially take $q=p$, $\theta^{\nu,>}\equiv 0$; when $p=p_\star$, recalling that $\theta^\nu_0\to \theta_0$ in $L^p$ by assumption \eqref{eq:approx_initial_data}, the claim follows from \cite[Lem. A.3]{BGM2025} (with any $q>p_\star$).
    
    By \eqref{eq:viscous_apriori2}, for any $\eps>0$, one can find a decomposition $\theta^\nu=\theta^{\nu,>}+\theta^{\nu,<}$ such that
    \begin{equation}\label{eq:uniform_decompositions}
        \sup_{\nu>0} \| \theta^{\nu,>}\|_{L^\infty_{\omega,t} L^p}\leq \eps, \qquad \sup_{\nu>0} \| \theta^{\nu,<}\|_{L^\infty_{\omega,t} L^q}\leq C_\eps.
    \end{equation}
    Let $\eps>0$ to be appropriately chosen later, such that the above decomposition holds.

    In the rest of this step, for notational simplicity we omit the time subscript (e.g.\ write $\theta$ instead of $\theta_s$).
    Recall the alternative rewriting of the nonlinearity $\cN^\beta(\theta)$ given by \eqref{eq:gSQG_nonlin_momentum}; let $u^\nu=-\nabla^\perp \Lambda^{-2+\beta}\theta^\nu$, $h^\nu=-\nabla^\perp \Lambda^{-2}\theta^\nu$, similarly for $\tilde \nu$, and set $u=u^\nu-u^{\tilde \nu}$, $h=h^\nu-h^{\tilde\nu}$. Then it holds
    \begin{align*}
        \langle \cN^\beta(\theta^\nu) -\cN^\beta(\theta^{\tilde\nu}), \Lambda^{-2} \xi\rangle
        & = \langle (u^\nu\cdot\nabla) h^\nu - (\nabla h^\nu)^T u^\nu
        - (u^{\tilde \nu}\cdot\nabla) h^{\tilde \nu} + (\nabla h^{\tilde \nu})^T u^{\tilde \nu}, - \nabla^\perp \Lambda^{-2} \xi\rangle\\
        & = \langle (u\cdot\nabla) h^\nu, h\rangle - \langle (\nabla h)^T u^\nu, h\rangle - \langle (\nabla h^{\tilde \nu})^T u, h\rangle
    \end{align*}
    where in the last passage we used the cancellation $\langle (u^{\tilde \nu}\cdot\nabla)h,h\rangle=0$ coming from $u^{\tilde \nu}$ being divergence free.
    The decomposition $\theta^\nu=\theta^{\nu,>}+\theta^{\nu,<}$ induces analogous ones $u^\nu= u^{\nu,>} + u^{\nu,<}$ and $h^\nu= h^{\nu,>} + h^{\nu,<}$, where $u^{\nu,\gtrless} \coloneqq - \nabla^ \perp \Lambda^{-2+\beta} \theta^{i,\gtrless}$ and $h^{\nu,\gtrless} \coloneqq - \nabla^ \perp \Lambda^{-2} \theta^{i,\gtrless}$; similarly for $\tilde \nu$.
    Accordingly, we further split the nonlinear terms as follows:
    \begin{equation}\label{eq:collection_nonlinear_terms}\begin{split}
        |\langle \cN^\beta(\theta^\nu) & -\cN^\beta(\theta^{\tilde\nu}), \Lambda^{-2} \xi\rangle|
        \leq I^> + I^< + J^> + J^< + K^> + K^<,\\
        I^> &\coloneqq |\langle h,  (u \cdot \nabla) h^{\nu,>} \rangle|, \qquad I^<\coloneqq |\langle  h, (u \cdot \nabla) h^{\nu,<}\rangle|,\\
        J^> &\coloneqq |\langle h,  (\nabla h^{\tilde\nu,>})^T  u  \rangle|,\qquad
        J^<\coloneqq |\langle  h, (\nabla h^{\tilde\nu,<})^T  u  \rangle|,\\
        K^> &\coloneqq |\langle  h, (\nabla h)^T u^{\nu,>}\rangle|,\qquad
        K^<\coloneqq |\langle h, (\nabla h)^T u^{\nu,<} \rangle|.		
    \end{split}\end{equation}
    Recalling the definition of all terms involved in function of $\theta^\nu$, $\theta^{\tilde \nu}$, $\theta$ and their decompositions, applying estimates \eqref{eq:tril1}-\eqref{eq:tril3} and \eqref{eq:uniform_decompositions}, we obtain
    \begin{align*}
        &I^> + J^> + K^>
        \leq \tilde C_1 \|\xi\|^2_{\dot H^{-\alpha}}(\|\theta^{\nu,>}\|_{L^{p_\star}} + \|\theta^{\tilde \nu,>}\|_{L^{p_\star}})
        \leq 2 \tilde C_1 \eps \|\xi\|_{\dot H^{-\alpha}}^2, \\
        &I^< + J^< + K^<
        \leq \tilde C_1 \| \xi \|_{\dot H^{-\alpha}}^{2-r_1-r_2} \| \xi \|_{\dot H^{-1}}^{r_1+r_2} (\| \theta^{1,<} \|_{L^q} + \| \theta^{2,<} \|_{L^q})
        \leq 2 \tilde C_1 C_\eps \| \xi \|_{\dot H^{-\alpha}}^{2-r_1-r_2} \| \xi \|_{\dot H^{-1}}^{r_1+r_2}, 
    \end{align*}
    for some $r_1$, $r_2\in [0,1)$, depending on $q$, such that $r_1+r_2>0$. 
    Choosing $\eps>0$ sufficiently small so that $8 \tilde C_1 \eps\leq \kappa_1$ then yields
    \begin{align*}
        |\langle \cN^\beta(\theta^\nu) -\cN^\beta(\theta^{\tilde\nu}), \Lambda^{-2} \xi\rangle| \leq \frac{\kappa_1}{4} \|\xi\|^2_{\dot H^{-\alpha}} + 2 \tilde C_1 C_\eps \| \xi \|_{\dot H^{-\alpha}}^{2-r_1-r_2} \| \xi \|_{\dot H^{-1}}^{r_1+r_2}
    \end{align*}
    which implies the desired estimate \eqref{eq:strong_exist_proof_eq1} after an application of Young's inequality
    (which is allowed since $r_1+r_2\in (0,2)$). 

    \textit{Step 4: Passage to the limit.}
    By Step 2, there exists a process $\theta$ such that $\theta^\nu\to \theta$ in $L^2_{\omega,t} (\dot H^{-\alpha}\cap \dot H^{-1})\cap L^1_\omega C_t \dot H^{-1}$, such that $\theta\vert_{t=0}=\theta_0$.
    Since $\theta^\nu$ were $\cF^W$-progressive processes, so is $\theta$.
    By properties of fractional operators, $u^\nu\to u$ in $ L^2_{\omega,t} \dot H^{1-\beta-\alpha}_x$, where $u=-\nabla^\perp \Lambda^{-2+\beta}\theta$. 
    Since $\theta^\nu_0\to \theta$ in $L^p$, the pathwise bound \eqref{eq:viscous_apriori2}  and weak compactness imply that
    \begin{align*}
        \sup_{\nu>0} \| \theta^\nu\|_{L^\infty_{\omega,t} L^p} \leq \sup_{\nu>0} \| \theta^\nu_0\|_{L^p}<\infty, \qquad  \theta^\nu\overset{\ast}\rightharpoonup \theta\ \text{ in } \ L^\infty_{\omega,t} L^p_x
    \end{align*}
    By assumption \eqref{eq:parameters_strong_existence}, $\alpha+\beta\leq 1$, therefore by Sobolev embeddings and weak-strong convergence we deduce that
    \begin{align*}
        u^\nu \theta^\nu\rightharpoonup u \theta\ \text{ in }\ L^2_{\omega,t} L^r_x \quad \text{where}\quad \frac{1}{r}=\frac{1}{p}+\frac{\alpha+\beta}{2}.
    \end{align*}
    Again by assumption  \eqref{eq:parameters_strong_existence}, it's easy to check that $r>1$;
    in particular this implies that $\cN^\beta(\theta^\nu)\to \cN^\beta(\theta)$ in $L^2_{\omega,t} L^1_\loc$.
    Similar arguments show that $\nu\Delta\theta^\nu\to 0$ in $L^2_{\omega,t} H^{-2}$, while $\theta^\nu \nabla\varphi\to \theta \nabla\varphi$ in $L^2_{\omega,t} \dot H^{-\alpha}$ for all $\varphi\in C^\infty_c$; aided by \cite[Lem. 2.8]{GalLuo2025}, we can therefore pass to the limit in the weak formulation of \eqref{eq:viscous_SPDE} to deduce that $\theta$ satisfies \eqref{eq:def_of_solution}.
    Overall this verifies that $\theta$ is a probabilistically strong solution, in the sense of \cref{def:solution}.

    It remains to verify that $\theta\in \cU^{p_\star}$. If $p>p_\star$, there is nothing to prove; for $p=p_\star$, similarly to \cite[Thm. 3.1]{BGM2025}, thanks to the uniform bounds \eqref{eq:uniform_decompositions}, one can pass to the limit (in weak-$\ast$ topologies) in the decompositions $\theta^\nu=\theta^{\nu,>}+\theta^{\nu,<}$, to find a similar decomposition $\theta=\theta^>+\theta^<$.   
\end{proof}

\begin{remark}
    Passing to the limit as $\tilde\nu\to 0^+$ in \eqref{eq:Cauchy_v1}, while keeping $\nu>0$ fixed, yields the estimate
    \begin{equation}\label{eq:convergence_rate_vv}
        \sup_{t\in [0,T]} \E[\| \theta^\nu_t-\theta_t\|_{\dot H^{-1}}^2] + \E\bigg[ \int_0^T \| \theta^\nu_t-\theta_t\|_{\dot H^{-\alpha}}^2 \dd t \bigg]
        \lesssim_{C_i,T}  \| \theta^\nu_0-\theta_0\|_{\dot H^{-1}}^2 + \nu^{2-\alpha} \| \theta^\nu_0\|_{L^2}^2.
    \end{equation}
    For suitable initial data, say $\theta_0\in L^1\cap L^\infty\cap \dot H^{-1}$, taking $\theta^\nu_0=\theta_0$ yields a polynomial rate of order $\nu^{2-\alpha}$.
    We leave the question of further improving either the rate or the norms involved for future research.  
\end{remark}

\section{Pathwise uniqueness} \label{sec:uniqueness}

We are now ready to complete the proofs of our main results; we start with the one concerning 2D Euler.

\begin{proof}[Proof of \cref{thm:main_Euler}]
Strong existence follows from \cref{thm:strong_existence}, so we only need to show pathwise uniqueness.
Let $\theta^1,\,\theta^2 \in L^\infty_{\omega,t} L ^p \cap L^2_{\omega, t} ( \dot H^{-1} \cap H^{-\alpha} )$ be two weak solutions to \eqref{eq:SPDEs_Ito}, with respect to the same tuple $(\Omega,\cF,\cF_t,\P,W)$ and starting from the same initial datum $\theta_0\in L^p$.
Without loss of generality, we may assume $\theta^1$ to be the strong solution given by \cref{thm:strong_existence} (which can be constructed on any probability space carrying $W$); indeed if we show that any other weak solution coincides with $\theta^1$, pathwise uniqueness follows. As a consequence, we have the additional information that $\theta\in \cU^{p_\star}$.

Let $\xi \coloneqq \theta^1-\theta^2$, $u \coloneqq u^1-u^2$.
As in \cref{thm:strong_existence}, we would like to proceed with a Gr\"onwall-type estimate for $\E[\| \xi_t\|^2_{\dot H^{-1}}]$. Due to poor regularity of solutions, we cannot directly apply It\^o formula to $\| \xi_t\|_{\dot H^{-1}}^2$ and we need to argue by regular approximations; specifically we look at $\langle G^\delta \ast \xi, \xi \rangle$, for $G$, $G^\delta$ as defined in \cref{subsec:covariance}.
Recalling the more convenient rewriting of the nonlinearity $\cN^0(\theta^i)$ given by \eqref{eq:Euler_nonlin_momentum}, after integration by parts one finds
\begin{equation}
    \begin{split}\label{eq:smooth_evolution}
		\dd\langle G^\delta \ast \xi_t, \xi_t \rangle =
		&  2 \langle \nabla^\perp G^\delta \ast \xi_t, (u^1_t \cdot \nabla) u^1_t- (u^2_t \cdot \nabla) u^2_t  \rangle \dd t\\
		& + \langle \trace [ Q  D^2G^\delta] \ast \xi_t, \xi_t \rangle \dd t + 2 \langle \xi_t \nabla G^\delta \ast \xi_t, \dd W_t \rangle;	
    \end{split}
\end{equation}
the term associated to $\trace [ Q  D^2G^\delta]$ is obtained combining the It\^o-Stratonovich correction and the It\^o second-order term as in \cite[Lem. 4.1]{coghi2023existence}.
Since $G^\delta$ is a Schwartz function,
\begin{align*}
    \|  (\nabla G^\delta \ast \xi_t) \, \xi_t \|_{H^{-1-\alpha}}
    \lesssim \| \nabla G^\delta \ast \xi_t\|_{C^2} \| \xi_t\|_{H^{-1-\alpha}} \lesssim_\delta \| \xi_t\|_{H^{-1-\alpha}}^2
\end{align*}
and so the stochastic integral appearing in \eqref{eq:smooth_evolution} is a martingale by \cite[Lem. 2.8]{GalLuo2025}.
Taking the expectation in \eqref{eq:smooth_evolution}, after some manipulations we obtain 
	\begin{align}
		\E \langle G^\delta \ast \xi_t, \xi_t \rangle =
		&  2 \E \int_0^t \langle \nabla^\perp G^\delta \ast \xi_s, (u_s \cdot \nabla) u^1_s \rangle  \dd s
		+ 2 \E \int_0^t \langle \nabla^\perp G^\delta \ast \xi_s, (u^2_s \cdot \nabla) u_s  \rangle \dd s \nonumber \\
		& + \E \int_0^t\langle \trace [ Q D^2G^\delta] \ast \xi_s, \xi_s \rangle \dd s. \label{eq:smooth_evolution_v2}
	\end{align}
    By construction and dominated convergence, the l.h.s. of \eqref{eq:smooth_evolution_v2} converges to $\E \| \xi_t\|_{\dot H^{-1}}^2$ as $\delta\to 0^+$; convergence of the last term instead is given by \cref{cor:reg_term_convergence}.
    Concerning the first two terms in the r.h.s. of \eqref{eq:smooth_evolution_v2}, convergence is granted by dominated convergence together with uniform-in-$\delta$ bounds coming from the application of \eqref{eq:tril1}. We provide details for the second one, the first being simpler.
    By \eqref{eq:def_varphi},  it holds $G^\delta \ast \xi = G \ast (\chi_\delta \ast \xi)\eqqcolon G \ast \xi^\delta$ with $|\widehat \chi_\delta (n)| \le 1$, so by \eqref{eq:tril1} we have
    \begin{equation}\label{eq:Youngs_ineq}\begin{split}
        |\langle \nabla^\perp G^\delta \ast \xi, (u^2 \cdot \nabla) u  \rangle| 
        &=|\langle \nabla^\perp G \ast \xi^\delta, (u^2 \cdot \nabla) u  \rangle|\\
        &\lesssim \| \xi^\delta \|_{\dot H^{-\alpha}}^{1-r_1} \| \xi^\delta \|_{\dot H^{-1}}^{r_1} \| \theta^2 \|_{\dot H^{-\alpha}}^{1-{r_2}} \| \theta^2 \|_{\dot H^{-1}}^{r_2}\| \xi \|_{L^p}\\
        &\le \| \xi \|_{\dot H^{-\alpha}}^{1-r_1} \| \xi \|_{\dot H^{-1}}^{r_1} \| \theta^2 \|_{\dot H^{-\alpha}}^{1-{r_2}} \| \theta^2 \|_{\dot H^{-1}}^{r_2}\| \xi \|_{L^p}\\
        & \lesssim \big(\| \xi \|_{\dot H^{-\alpha}}^{2} + \| \xi \|_{\dot H^{-1}}^{2} + \| \theta^2 \|_{\dot H^{-\alpha}}^{2} + \| \theta^2 \|_{\dot H^{-1}}^{2}\big) \| \xi \|_{L^p}
    \end{split}\end{equation}
    where in the last step we applied (multiple times) Young's inequality.
    By our assumptions on $\theta^i$,  the last term in \eqref{eq:Youngs_ineq} belongs to $L^1_{\omega, t}$ and provides an explicit dominant for $|\langle \nabla^\perp G^\delta \ast \xi, (u^2 \cdot \nabla) u  \rangle|$.
    Passing to the limit as $\delta\to 0^+$ in \eqref{eq:smooth_evolution_v2}, we obtain
	\begin{equation}\label{eq:ineq_balance}\begin{split}
		\E \| \xi_t\|^2_{\dot H^{-1}} =
		&  -2 \E \int_0^t \langle u_s, (u_s \cdot \nabla) u^1_s \rangle  \dd s  + \E \int_0^t \int_{\R^2} F(n) |\widehat \xi_s(n)|^2 \dd n \dd s\\
        & \leq -2 \E \int_0^t \langle u_s, (u_s \cdot \nabla) u^1_s \rangle  \dd s - \kappa_1 \int_0^t \| \xi_s\|_{\dot H^{-\alpha}}^2 \dd s + \kappa_2 \int_0^t \| \xi_s\|_{\dot H^{-1}}^2 \dd s
	\end{split}\end{equation}
    where in the first line we used that $\langle \nabla^\perp G \ast \xi_s, (u^2_s \cdot \nabla )u_s  \rangle =  -\langle u_s, (u^2_s \cdot \nabla )u_s  \rangle =0$, since $u_2$ is divergence free, and in the second one we applied \eqref{eq:regularization_convergence}.
    
    Note that in \eqref{eq:ineq_balance} $u^2$ does not appear anymore; this allows to perform a weak-strong uniqueness type argument, exploiting the additional information on the strong solution $\theta^1$ coming from \cref{thm:strong_existence}, in the same vein as \cite{BGM2025}.
    For $\eps>0$ to be fixed later, consider a decomposition $\theta^1= \theta^{1,>} + \theta^{1,<}$ with $\| \theta^{1,>} \|_{L^\infty_{\omega, t} L^{p_\star}}\leq \eps$ and $\| \theta^{1,<} \|_{L^\infty_{\omega, t} L^{q}}\leq C_\eps$ for some $q \in (p_\star, \infty)$; then we have $u^1= u^{1,>} + u^{1,<}$, with $u^{1,\gtrless} \coloneqq - \nabla^ \perp \Lambda^{-2} \theta^{1,\gtrless}$.
    Similarly to \eqref{eq:collection_nonlinear_terms}, we can split the nonlinear term and employ the estimate \eqref{eq:tril1} to find
    \begin{align*}
        |\langle u, (u \cdot \nabla) u^1\rangle |
        & \le  |\langle u, (u \cdot \nabla) u^{1,>}\rangle | + |\langle u, (u \cdot \nabla) u^{1,<}\rangle|\\
        & \lesssim \| \xi \|^2_{\dot H^{-\alpha}} \| \theta^{1,>} \|_{L^{p_\star}} + \| \xi \|_{\dot H^{-\alpha}}^{2-{r_1}-r_2} \| \xi \|_{\dot H^{-1}}^{r_1+r_2}\| \theta^{1,<} \|_{L^q}\\
        & \lesssim \eps \| \xi \|^2_{\dot H^{-\alpha}} + C_\eps \| \xi \|_{\dot H^{-\alpha}}^{2-{r_1}-r_2} \| \xi \|_{\dot H^{-1}}^{r_1+r_2}.
    \end{align*}
    We can then choose $\eps>0$ small enough so to reabsorb the first term using $-\kappa_1 \| \xi\|_{\dot H^{-\alpha}}^2$ and use Young's inequality on the second one (since $r_1+r_2\in (0,2)$); inserting such estimates in \eqref{eq:smooth_evolution_v2}, one arrives at
    \begin{equation*}		
		\E \| \xi_t\|^2_{\dot H^{-1}} +\frac{\kappa_1}{2}\, \E \int_0^t \|\xi_s \|^2_{\dot H^{-\alpha}}\dd s 
		\le \tilde C\, \E \int_0^t \|\xi_s \|^2_{\dot H^{-1}}\dd s.
	\end{equation*}
    Pathwise uniqueness readily follows by Gr\"onwall's inequality.
\end{proof}

\begin{proof}[Proof of \cref{thm:main_gSQG}]
    The proof is a slight variant on both \cref{thm:main_Euler,thm:strong_existence}, so we mostly sketch it, highlighting the main differences.
    
    Let $\theta^i\in L^\infty_{\omega,t} L ^p \cap L^2_{\omega, t} ( \dot H^{-1} \cap \dot H^{-\alpha} )\cap \cU^{p_\star}$ be two weak solutions, defined on the same probability space, driven by the same $W$ and with same initial datum $\theta_0$; set $u^i=-\nabla^\perp \Lambda^{\beta-2}\theta^i$, $h^i=-\nabla^\perp \Lambda^{-2}\theta^i$, $u=u^1-u^2$, $h=h^1-h^2$.
    We perform the same strategy of \cref{thm:main_Euler}, based on deriving the balance for $\E\|\xi_t\|_{\dot H^{-1}}^2$ by first looking at $\E \langle G^\delta\ast \xi_t,\xi_t\rangle$ and then passing to the limit. Similarly to \cref{thm:strong_existence}, this time we express the nonlinear terms $\cN^\beta(\theta^i)$ using \eqref{eq:gSQG_nonlin_momentum}.
    Using appropriate bounds from \cref{lem:trilinear_estimates}, dominated convergence and \cref{cor:reg_term_convergence}, taking the limit $\delta\to 0^+$ eventually one arrives at
    \begin{equation}\label{eq:balance_gSQG_v1}
		\E \| \xi_t\|^2_{\dot H^{-1}} + \kappa_1 \int_0^t \| \xi_s\|^2_{\dot H^{-\alpha}} \dd s \leq \kappa_2 \int_0^t \| \xi_s\|^2_{\dot H^{-1}} \dd s + 2 \int_0^t \E\big[ I_s + J_s + K_s\big] \dd s
    \end{equation}
    where
    \begin{align*}
		I_s \coloneqq |\langle  h_s, (u_s \cdot \nabla) h^1_s \rangle|,\quad
        J_s \coloneqq |\langle  h_s,  (\nabla h^2_s)^T  u_s  \rangle|, \quad
        K_s \coloneqq |\langle  h_s, (\nabla h_s)^T u^1_s \rangle  |;
	\end{align*}
    we used the cancellation of one contribution: $\langle h_s, (u^2_s \cdot \nabla) h_s  \rangle =0$ since $u^1$ is divergence free.
    
    The terms $I_s$ and $K_s$ depend explicitly on $\theta^1$, while $J_s$ depends on $\theta^2$, which prevents the use of the same weak-strong uniqueness argument as in \cref{thm:main_Euler}; this is why uniqueness is restricted to solutions $\theta^i$ additionally belonging to $\cU^{p_\star}$, i.e. enjoying nice decompositions $\theta^i=\theta^{i,>}+\theta^{i,<}$ with $\theta^{i,<}\in L^\infty_{\omega,t} L^{q_i}$ for some $q_i\in (p_\star,\infty)$.
    For notational simplicity, below we will assume $q_1=q_2$, although it's not necessary. 
    Decomposing accordingly $u^i$ and $h^i$, one can perform similar bounds as in \eqref{eq:collection_nonlinear_terms} (with $\theta^1$ playing the role of $\theta^\nu$, $\theta^2$ of $\theta^{\tilde \nu}$); choosing $\eps>0$ small enough, inserting these estimates in \eqref{eq:balance_gSQG_v1}, one arrives at
    \begin{align*}
		\E \| \xi_t\|^2_{\dot H^{-1}} + \frac{\kappa_1}{2} \int_0^t \| \xi_s\|^2_{\dot H^{-\alpha}} \dd s
        \lesssim \int_0^t \| \xi_s\|^2_{\dot H^{-1}} \dd s +
        \int_0^t \E\big[ \| \xi \|_{\dot H^{-\alpha}}^{2-{r_1}- r_2} \| \xi \|_{\dot H^{-1}}^{r_1+ r_2}\big] \dd s
    \end{align*}
    A further application of Young's inequality and Gr\"onwall's lemma provide the conclusion.
\end{proof}

\appendix
\section{Technical lemmas}\label{app:tech_lem}

\begin{lemma}\label{lem:approximation}
    Let $p\in (1,\infty)$, $f\in L^p\cap \dot H^{-1}$. Then there exists a family $\{f^\eps\}_{\eps>0}\subset C^\infty_c$ which is bounded in $L^p\cap \dot H^{-1}$ such that $f^\eps\to f$ therein.
    Moreover, uniformly in $p$, we have the estimate
    \begin{equation}\label{eq:approx_estim}
        \| f^\eps\|_{L^2} \lesssim \eps^{-1} \| f\|_{\dot H^{-1}}
    \end{equation}
\end{lemma}

\begin{proof}
    Similar approximations schemes have previously appeared e.g. in \cite[Sec. 3.2]{coghi2023existence}, but we provide a self-contained proof.
    
    Let $\psi$ be a smooth function such that $\psi\equiv 1$ for $|x|<1$, $\psi\equiv 0$ for $|x|\geq 2$; set $\psi^\eps(x):= \psi(\eps x)$. Let $\varphi\in C^\infty_c$ be a probability density, denote by $\varphi^\eps$ the associated standard mollifiers. Set
    \begin{equation*}
        f^\eps:= \varphi^\eps \ast \tilde f^\eps, \quad \tilde f^\eps:=-\nabla\cdot (\psi^\eps \nabla \Lambda^{-2} f)=\psi^\eps f - \nabla\psi^\eps\cdot\nabla\Lambda^{-2} f
    \end{equation*}
    Clearly $\tilde f^\eps$ are compactly supported distributions; if we show that $\tilde f^\eps\to f$ in $L^p\cap \dot H^{-1}
    $, the first part of the statement follows by properties of mollifiers. It holds
    \begin{align*}
        \| \tilde f^\eps - f\|_{\dot H^{-1}}
        = \| \nabla\cdot ((1-\psi^\eps) \nabla \Lambda^{-2} f\|_{\dot H^{-1}}\leq \| (1-\psi^\eps) \nabla \Lambda^{-2} f\|_{L^2}
    \end{align*}
    and the last quantity converges to $0$ by dominated convergence since $\nabla \Lambda^{-2} f\in L^2$. The same argument shows that
    \begin{align*}
        \sup_{\eps>0} \|\tilde f^\eps\|_{\dot H^{-1}}\leq \| f\|_{\dot H^{-1}}
    \end{align*}
    which by properties of mollifiers implies
    \begin{align*}
        \| f^\eps\|_{L^2} \lesssim \| \nabla\varphi^\eps\|_{L^1} \| \tilde f^\eps\|_{\dot H^{-1}}\lesssim \eps^{-1} \| f\|_{\dot H^{-1}}
    \end{align*}
    proving \eqref{eq:approx_estim}.
    It's clear that $\psi^\eps f\to f$ in $L^p$, so it remains to show that $\nabla\psi^\eps\cdot\nabla \Lambda^{-2} f\to 0$ in $L^p$. We divide it in two cases.

    \textit{Case $p\in (1,2]$:} By assumption and properties of Calder\'on-Zygmund operators, $\nabla\Lambda^{-2}f\in L^2\cap \dot W^{1,p}\hookrightarrow L^q$ for all $q\in [2,q_\ast)$, $q_\ast=2p/(2-p)>2$ being the exponent associated to Sobolev embedding. On the other hand, $\| \nabla\psi^\eps\|_{L^r}\sim \eps^{1-2/r}$ converges to $0$ for all $r\in (2,\infty)$. So in this case conclusion follows by H\"older's inequality.
    
    \textit{Case $p\in (2,\infty)$:} In this case embeddings yield $\nabla\Lambda^{-2} f\in L^2\cap \dot W^{1,p}\hookrightarrow L^p$ and conclusion follows from $\|\nabla\psi^\eps\|_{L^\infty}\sim \eps\to 0$.
\end{proof}

\begin{proof}[Proof of \cref{lem:anom_reg}]
	The proof is achieved by combining results from \cite{coghi2023existence, JiaoLuo2025_Euler}. Following \cite[eq. (7.5)]{coghi2023existence}, we split
	\begin{align*}
		\trace\left[Q(x)D^2G^{\delta}(x)\right]
		&=  \trace\left[\left(Q(x)\right)D^2G(x)\right] \varphi(x)
         - \trace\left[Q(x)D^2(G-G^{\delta})(x)\right]\varphi(x) \nonumber\\
		& \quad + \trace[Q(x) D^2G^{\delta}(x)](1-\varphi(x))
		\nonumber\\
		& =:A+R_2^\delta+R_3^\delta
	\end{align*}
    where $\varphi(x)=\varphi(|x|)$ is a radial $C^\infty$ function such that $0\le \varphi\le 1$, $\varphi(x)=1$ for $|x|\le 1$ and $\varphi(x)=0$ for $|x|\ge 2$;
    in the above, we can identify the distribution $D^2G$ with the pointwise-defined second derivative of $G$, since its non-integrable contribution at the origin (which would result in a Dirac term) is compensated by $|Q(x)|\lesssim |x|^{2\alpha}$. Then, \cite[Lem. 4.3]{coghi2023existence} guarantees that
	\begin{equation*}
		\widehat{A}(n) \le - c_1\langle n\rangle^{-2\alpha} +C_1\langle n\rangle^{-2}
	\end{equation*}  
	and from its proof it is also clear that $|\widehat{A}(n)| \lesssim\langle n\rangle^{-2\alpha} +\langle n\rangle^{-2}.$
	Concerning the term $R_2^\delta$, \cite[Lem. 3.1]{JiaoLuo2025_Euler} shows that $|\widehat{R}^\delta_2(n)|\le C_2 \langle n\rangle^{-2\alpha}$,	while $R_3^\delta$ is addressed in \cite[Lem. 4.5]{coghi2023existence}, where it is established  that $|\widehat{R}^\delta_3(n)| \le C_3 |n|^{-2}$.
\end{proof}

\begin{proof}[Proof of \cref{cor:reg_term_convergence}]
Let $F^\delta\coloneq c\cF (\trace [ Q D^2G^\delta])$, where $c$ is a dimensional constant depending on the convention of Fourier transform, so that
\begin{align*}
    \int_0^t \langle \trace [ Q D^2G^\delta]\ast \xi_s, \xi_s\rangle \dd s
    = \int_0^t \int_{\R^d} F^\delta(n)|\hat\xi_s(n)|^2 \dd n \dd s;
\end{align*}
basic manipulations with Fourier transform (cf. \cite[Sec. 3]{GGM2024}) and formula \eqref{eq:def_varphi} yield
\begin{align*}
    F^\delta(n)=  c \int_{\R^2} \langle n-k \rangle^{-2-2\alpha} |P^\perp_{n-k} n|^2 \left(\hat\chi_\delta(k)|k|^{-2} - \hat\chi_\delta(n)|n|^{-2} \right)\dd k;
\end{align*}
since $\hat\chi_\delta(k)\to 1$ pointwise and $|\hat\chi_\delta(k)|\leq 1$, dominated convergence implies that $F^\delta(n)\to F(n)$ for every $n\neq 0$. On the other hand, by \cref{lem:anom_reg} we have the pointwise estimate
\begin{align*}
    \sup_{\delta>0} |F^\delta(n)-F(n)|\lesssim |n|^{-2\alpha}+|n|^{-2}.
\end{align*}
In particular, $|F^\delta(n)-F(n)||\hat \xi_s(n)|^2$ is dominated by a multiple of $(|n|^{-2\alpha}+|n|^{-2})|\hat \xi_s(n)|^2$, where
\begin{align*}
    \int_0^T \int_{\R^2} (|n|^{-2\alpha}+|n|^{-2})|\hat\xi_s(n)|^2 \dd n \dd s
    = \| \xi\|_{L^2_t (\dot H^{-1}\cap \dot H^{-\alpha})}^2.
\end{align*}
As the last random variable is integrable by assumption, it is $\P$-a.s. finite.
In particular, at any fixed $\omega\in\Omega$ such that $\| \xi(\omega)\|_{L^2_t (\dot H^{-1}\cap \dot H^{-\alpha})}<\infty$, by dominated convergence (applied to the integral in $(s,n)$) it holds that
\begin{align*}
    \sup_{t\in [0,T]} \bigg|\int_0^t \langle \trace [ Q D^2G^\delta]\ast \xi_s, \xi_s\rangle \dd s & - \int_0^t \int_{\R^2} F(n) |\hat \xi_s(n)|^2 \dd n \dd s\bigg|\\
    & \leq \int_0^T \int_{\R^2} |F^\delta(n)-F(n)| |\hat\xi_s(n)|^2 \dd n \dd s\to 0 \quad \text{as }\ \delta\to 0^+
\end{align*}
which proves the $\P$-a.s. uniform limit.
The inequality in \eqref{eq:regularization_convergence} then follows from \eqref{eq:fundamental_estimate_covariance}.
\end{proof}

%----------------------------------------------
\textbf{Acknowledgements.}
    The authors acknowledge support from the Istituto Nazionale di Alta Matematica (INdAM) through the project GNAMPA 2025 “Modelli stocastici in Fluidodinamica e Turbolenza”. We are grateful to Mario Maurelli for useful discussion while working on this project.
    \\
%----------------------------------------------

\textbf{Conflict of interest statement.} On behalf of all authors, the corresponding author states that there is no conflict of interest.

\textbf{Data availability statement.} This manuscript has no associated data.

\bibliography{biblio.bib}{}

@article {FlandoliPappalettera2021,
    AUTHOR = {Flandoli, Franco and Pappalettera, Umberto},
     TITLE = {2{D} {E}uler equations with {S}tratonovich transport noise as
              a large-scale stochastic model reduction},
   JOURNAL = {J. Nonlinear Sci.},
  FJOURNAL = {Journal of Nonlinear Science},
    VOLUME = {31},
      YEAR = {2021},
    NUMBER = {1},
     PAGES = {Paper No. 24, 38},
      ISSN = {0938-8974,1432-1467},
   MRCLASS = {76B03 (35Q31 76M35)},
  MRNUMBER = {4215938},
MRREVIEWER = {Francesco\ Fanelli},
       DOI = {10.1007/s00332-021-09681-w},
       URL = {https://doi.org/10.1007/s00332-021-09681-w},
}

@article{crippa2025zero,
  title={Zero-noise selection {and Large Deviations in $L^\infty_t L^p_x $ for the stochastic transport equation beyond DiPerna-Lions}},
  author={Crippa, Gianluca and Luongo, Eliseo and Pappalettera, Umberto},
  journal={arXiv:2506.06947},
  year={2025}
}

@article{bagnara2024anomalous,
  title={Anomalous {Regularization in Kazantsev-Kraichnan Model}},
  author={Bagnara, Marco and Grotto, Francesco and Maurelli, Mario},
  journal={arXiv:2411.09482},
  year={2024}
}

@article {Sch2013,
    AUTHOR = {Scheutzow, Michael},
     TITLE = {A stochastic {G}ronwall lemma},
   JOURNAL = {Infin. Dimens. Anal. Quantum Probab. Relat. Top.},
  FJOURNAL = {Infinite Dimensional Analysis, Quantum Probability and Related
              Topics},
    VOLUME = {16},
      YEAR = {2013},
    NUMBER = {2},
     PAGES = {1350019, 4},
       DOI = {10.1142/S0219025713500197},
       URL = {https://doi.org/10.1142/S0219025713500197},
}

@article {GeiSch2021,
    AUTHOR = {Geiss, Sarah and Scheutzow, Michael},
     TITLE = {Sharpness of {L}englart's domination inequality and a sharp
              monotone version},
   JOURNAL = {Electron. Commun. Probab.},
  FJOURNAL = {Electronic Communications in Probability},
    VOLUME = {26},
      YEAR = {2021},
     PAGES = {Paper No. 44, 8},
       DOI = {10.1214/21-ecp413},
       URL = {https://doi-org.univaq.idm.oclc.org/10.1214/21-ecp413},
}

@article {BFM2016,
    AUTHOR = {Brze\'zniak, Zdzis{\l}aw and Flandoli, Franco and Maurelli,
              Mario},
     TITLE = {Existence and uniqueness for stochastic 2{D} {E}uler flows
              with bounded vorticity},
   JOURNAL = {Arch. Ration. Mech. Anal.},
  FJOURNAL = {Archive for Rational Mechanics and Analysis},
    VOLUME = {221},
      YEAR = {2016},
    NUMBER = {1},
     PAGES = {107--142},
       DOI = {10.1007/s00205-015-0957-8},
       URL = {https://doi-org.univaq.idm.oclc.org/10.1007/s00205-015-0957-8},
}

@article{rowan2025obukhov,
  title={The {Obukhov--Corrsin spectrum of passive scalar turbulence through anomalous regularization}},
  author={Rowan, Keefer},
  journal={arXiv:2512.02853},
  year={2025}
}

@article{drivas2025anomalous,
  title={Anomalous dissipation and regularization in isotropic {G}aussian turbulence},
  author={Drivas, Theodore D. and Galeati, Lucio and Pappalettera, Umberto},
  journal={arXiv:2509.10211},
  year={2025}
}

@article{zhao2024onsager,
  title={An {Onsager-type Theorem for General 2D Active Scalar Equations}},
  author={Zhao, Xuanxuan},
  journal={arXiv:2412.11094},
  year={2024}
}

@article{castro2025unstable,
  title={Unstable vortices, sharp non-uniqueness with forcing, and global smooth solutions for the {SQG} equation},
  author={Castro, {\'A}ngel and Faraco, Daniel and Mengual, Francisco and Solera, Marcos},
  journal={arXiv:2502.10274},
  year={2025}
}

@book {CFG2008,
    AUTHOR = {Cardy, John and Falkovich, Gregory and Gawedzki, Krzysztof},
     TITLE = {Non-equilibrium statistical mechanics and turbulence},
    SERIES = {London Mathematical Society Lecture Note Series},
    VOLUME = {355},
 PUBLISHER = {Cambridge University Press},
      YEAR = {2008},
     PAGES = {x+161},
       DOI = {10.1017/CBO9780511812149},
       URL = {https://doi.org/10.1017/CBO9780511812149},
}

@incollection {BofEck2012,
    AUTHOR = {Boffetta, Guido and Ecke, Robert E.},
     TITLE = {Two-dimensional turbulence},
 BOOKTITLE = {Annual review of fluid mechanics},
    SERIES = {Annu. Rev. Fluid Mech.},
    VOLUME = {44},
     PAGES = {427--451},
 PUBLISHER = {Annual Reviews, Palo Alto, CA},
      YEAR = {2012},
       DOI = {10.1146/annurev-fluid-120710-101240},
       URL = {https://doi-org.univaq.idm.oclc.org/10.1146/annurev-fluid-120710-101240},
}

@article{JiaoLuo2025_Boussinesq,
  title={S{trong Uniqueness by Kraichnan Transport Noise for the 2D Boussinesq Equations with Zero Viscosity}},
  author={Jiao, Shuaijie and Luo, Dejun},
  journal={arXiv:2504.19153},
  year={2025}
}

@article {BGM2025,
    AUTHOR = {Bagnara, Marco and Galeati, Lucio and Maurelli, Mario},
     TITLE = {Regularization by rough {K}raichnan noise for the generalised
              {SQG} equations},
   JOURNAL = {Math. Ann.},
  FJOURNAL = {Mathematische Annalen},
    VOLUME = {392},
      YEAR = {2025},
    NUMBER = {4},
     PAGES = {4773--4830},
      ISSN = {0025-5831,1432-1807},
   MRCLASS = {35Q35 (60H15 60H50)},
  MRNUMBER = {4958490},
       DOI = {10.1007/s00208-025-03211-9},
       URL = {https://doi.org/10.1007/s00208-025-03211-9},
}

@article {cordoba2025strong,
    AUTHOR = {C\'ordoba, Diego and Lucas-Manch\'on, Jos\'e{} and
              Mart\'inez-Zoroa, Luis},
     TITLE = {Strong ill-posedness and non-existence in {S}obolev spaces for
              generalized-{SQG}},
   JOURNAL = {Nonlinearity},
  FJOURNAL = {Nonlinearity},
    VOLUME = {38},
      YEAR = {2025},
    NUMBER = {8},
     PAGES = {Paper No. 085021, 57},
      ISSN = {0951-7715,1361-6544},
   MRCLASS = {35R25},
  MRNUMBER = {4950574},
       DOI = {10.1088/1361-6544/adf9b1},
       URL = {https://doi.org/10.1088/1361-6544/adf9b1},
}

@article{brue2024flexibility,
  title={F{lexibility of Two-Dimensional Euler Flows with Integrable Vorticity}},
  author={Bru{\`e}, Elia and Colombo, Maria and Kumar, Anuj},
  journal={arXiv:2408.07934},
  year={2024}
}

@book {ABCDGJK2024,
    AUTHOR = {Albritton, Dallas and Bru\'{e}, Elia and Colombo, Maria and De Lellis, Camillo and Giri, Vikram and Janisch, Maximilian and Kwon, Hyunju},
     TITLE = {Instability and non-uniqueness for the 2{D} {E}uler equations,
              after {M}. {V}ishik},
    SERIES = {Annals of Mathematics Studies},
    VOLUME = {219},
 PUBLISHER = {Princeton University Press},
      YEAR = {2024},
     PAGES = {ix+136},
  MRNUMBER = {4729618},
}

@book {RozLot2018,
    AUTHOR = {Rozovsky, Boris L. and Lototsky, Sergey V.},
     TITLE = {Stochastic evolution systems},
    SERIES = {Probability Theory and Stochastic Modelling},
    VOLUME = {89},
   EDITION = {Second},
 PUBLISHER = {Springer, Cham},
      YEAR = {2018},
     PAGES = {xvi+330},
  MRNUMBER = {3839316},
       DOI = {10.1007/978-3-319-94893-5},
       URL = {https://doi.org/10.1007/978-3-319-94893-5},
}

@article{vishik2018a,
  title={Instability and non-uniqueness in the {C}auchy problem for the {E}uler equations of an ideal incompressible fluid. {P}art {I}},
  author={Vishik, Misha},
  journal={arXiv:1805.09426},
  year={2018}
}

@article{vishik2018b,
  title={Instability and non-uniqueness in the {C}auchy problem for the {E}uler equations of an ideal incompressible fluid. {P}art {II}},
  author={Vishik, Misha},
  journal={arXiv:1805.09440},
  year={2018}
}

@article {castro2024,
    AUTHOR = {Castro, \'Angel and Faraco, Daniel and Mengual, Francisco and
              Solera, Marcos},
     TITLE = {A proof of {V}ishik's nonuniqueness theorem for the forced
              2{D} {E}uler equation},
   JOURNAL = {J. Reine Angew. Math.},
  FJOURNAL = {Journal f\"ur die Reine und Angewandte Mathematik. [Crelle's Journal]},
    VOLUME = {824},
      YEAR = {2025},
     PAGES = {253--288},
       DOI = {10.1515/crelle-2025-0025},
       URL = {https://doi.org/10.1515/crelle-2025-0025},
}

@article {Holm2015,
    AUTHOR = {Holm, Darryl D.},
     TITLE = {Variational principles for stochastic fluid dynamics},
   JOURNAL = {Proc. A.},
  FJOURNAL = {Proceedings A},
    VOLUME = {471},
      YEAR = {2015},
    NUMBER = {2176},
     PAGES = {20140963, 19},
  MRNUMBER = {3325187},
       DOI = {10.1098/rspa.2014.0963},
       URL = {https://doi.org/10.1098/rspa.2014.0963},
}

@article {BrCaFl1991,
    AUTHOR = {Brze\'{z}niak, Z. and Capi\'{n}ski, M. and Flandoli, F.},
     TITLE = {Stochastic partial differential equations and turbulence},
   JOURNAL = {Math. Models Methods Appl. Sci.},
  FJOURNAL = {Mathematical Models and Methods in Applied Sciences},
    VOLUME = {1},
      YEAR = {1991},
    NUMBER = {1},
     PAGES = {41--59},
  MRNUMBER = {1105007},
       DOI = {10.1142/S0218202591000046},
       URL = {https://doi.org/10.1142/S0218202591000046},
}

@phdthesis{ma2022some,
    title={Some Onsager’s conjecture type results for a family of odd active scalar equations},
    author={Ma, Andrew G.},
    year={2022},
    DOI = {http://dx.doi.org/10.26153/tsw/42129},
}

@article {JeoKim2024,
    AUTHOR = {Jeong, In-Jee and Kim, Junha},
     TITLE = {Strong ill-posedness for {SQG} in critical {S}obolev spaces},
   JOURNAL = {Anal. PDE},
  FJOURNAL = {Analysis \& PDE},
    VOLUME = {17},
      YEAR = {2024},
    NUMBER = {1},
     PAGES = {133--170},
  MRNUMBER = {4702316},
       DOI = {10.2140/apde.2024.17.133},
       URL = {https://doi.org/10.2140/apde.2024.17.133},
}

@article {JiaoLuo2025_Euler,
	AUTHOR = {Jiao, Shuaijie and Luo, Dejun},
	TITLE = {On the pathwise uniqueness of stochastic 2{D} {E}uler
	equations with {K}raichnan noise and {$L^p$}-data},
	JOURNAL = {J. Math. Fluid Mech.},
	FJOURNAL = {Journal of Mathematical Fluid Mechanics},
	VOLUME = {27},
	YEAR = {2025},
	NUMBER = {3},
	PAGES = {Paper No. 38, 9},
	ISSN = {1422-6928,1422-6952},
	MRCLASS = {76M35},
	MRNUMBER = {4905596},
	DOI = {10.1007/s00021-025-00943-1},
	URL = {https://doi.org/10.1007/s00021-025-00943-1},
}

@article {JiaoLuo2025_mSQG,
	AUTHOR = {Jiao, Shuaijie and Luo, Dejun},
	TITLE = {Well-posedness of stochastic m{SQG} equations with {K}raichnan
	noise and {$L^p$} data},
	JOURNAL = {J. Differential Equations},
	FJOURNAL = {Journal of Differential Equations},
	VOLUME = {438},
	YEAR = {2025},
	PAGES = {Paper No. 113362, 41},
	ISSN = {0022-0396,1090-2732},
	MRCLASS = {35R60 (35Q35 35Q86 60H15)},
	MRNUMBER = {4899753},
	DOI = {10.1016/j.jde.2025.113362},
	URL = {https://doi.org/10.1016/j.jde.2025.113362},
}

@article {choi2024well,
    AUTHOR = {Choi, Young-Pil and Jung, Jinwook and Kim, Junha},
     TITLE = {On well/ill-posedness for the generalized surface quasi-geostrophic equation in {H}\"older spaces},
   JOURNAL = {J. Differential Equations},
  FJOURNAL = {Journal of Differential Equations},
    VOLUME = {443},
      YEAR = {2025},
     PAGES = {Paper No. 113521, 36},
       DOI = {10.1016/j.jde.2025.113521},
       URL = {https://doi-org.univaq.idm.oclc.org/10.1016/j.jde.2025.113521},
}

@book {Bahouri2011,
    AUTHOR = {Bahouri, Hajer and Chemin, Jean-Yves and Danchin, Rapha\"{e}l},
     TITLE = {Fourier analysis and nonlinear partial differential equations},
    SERIES = {Grundlehren Math. Wiss.},
    VOLUME = {343},
 PUBLISHER = {Springer, Heidelberg},
      YEAR = {2011},
     PAGES = {xvi+523},
       DOI = {10.1007/978-3-642-16830-7},
       URL = {https://doi.org/10.1007/978-3-642-16830-7},
}

@article {GraOh2014,
    AUTHOR = {Grafakos, Loukas and Oh, Seungly},
     TITLE = {The {K}ato-{P}once inequality},
   JOURNAL = {Comm. Partial Differential Equations},
  FJOURNAL = {Communications in Partial Differential Equations},
    VOLUME = {39},
      YEAR = {2014},
    NUMBER = {6},
     PAGES = {1128--1157},
  MRNUMBER = {3200091},
       DOI = {10.1080/03605302.2013.822885},
       URL = {https://doi.org/10.1080/03605302.2013.822885},
}

@article {BuShVi2019,
    AUTHOR = {Buckmaster, Tristan and Shkoller, Steve and Vicol, Vlad},
     TITLE = {Nonuniqueness of weak solutions to the {SQG} equation},
   JOURNAL = {Comm. Pure Appl. Math.},
  FJOURNAL = {Communications on Pure and Applied Mathematics},
    VOLUME = {72},
      YEAR = {2019},
    NUMBER = {9},
     PAGES = {1809--1874},
  MRNUMBER = {3987721},
       DOI = {10.1002/cpa.21851},
       URL = {https://doi.org/10.1002/cpa.21851},
}

@article {Rowan2023,
    AUTHOR = {Rowan, Keefer},
     TITLE = {On anomalous diffusion in the {K}raichnan model and
              correlated-in-time variants},
   JOURNAL = {Arch. Ration. Mech. Anal.},
  FJOURNAL = {Archive for Rational Mechanics and Analysis},
    VOLUME = {248},
      YEAR = {2024},
    NUMBER = {5},
     PAGES = {Paper No. 93, 47},
      ISSN = {0003-9527,1432-0673},
   MRCLASS = {76M35 (35K57 60K50 76F35)},
  MRNUMBER = {4801699},
MRREVIEWER = {Adrian\ Muntean},
       DOI = {10.1007/s00205-024-02045-0},
       URL = {https://doi.org/10.1007/s00205-024-02045-0},
}

@article {GalLuo2025,
    AUTHOR = {Galeati, Lucio and Luo, Dejun},
     TITLE = {Weak well-posedness by transport noise for a class of 2{D}
              fluid dynamics equations},
   JOURNAL = {J. Funct. Anal.},
  FJOURNAL = {Journal of Functional Analysis},
    VOLUME = {289},
      YEAR = {2025},
    NUMBER = {12},
     PAGES = {Paper No. 111158, 59},
       DOI = {10.1016/j.jfa.2025.111158},
       URL = {https://doi-org.univaq.idm.oclc.org/10.1016/j.jfa.2025.111158},
}

@article {GalLuo2023,
    AUTHOR = {Galeati, Lucio and Luo, Dejun},
     TITLE = {Weak well-posedness by transport noise for a class of 2{D}
              fluid dynamics equations},
   JOURNAL = {arXiv:2305.08761, extended version of the published manuscript},
      YEAR = {2025},
       DOI = {
https://doi.org/10.48550/arXiv.2305.08761},
       URL = {https://arxiv.org/abs/2305.08761},
}

@article {GGM2024,
      title={Anomalous {R}egularization in {K}raichnan's {P}assive {S}calar {M}odel}, 
      author={Galeati, Lucio and Grotto, Francesco and Maurelli, Mario},
  journal={arXiv:2407.16668},
      year={2024},
}

@article{Kraichnan1968,
  title={Small-scale structure of a scalar field convected by turbulence},
  author={Kraichnan, R. H.},
  journal={The Physics of Fluids},
  volume={11},
  number={5},
  pages={945--953},
  year={1968},
  publisher={American Institute of Physics}
}

@article{coghi2023existence,
  title={Existence and uniqueness by {K}raichnan noise for 2{D} {E}uler equations with unbounded vorticity},
  author={Coghi, Michele and Maurelli, Mario},
  journal={arXiv:2308.03216},
  year={2023}
}

@article {zelati2023statistically,
    AUTHOR = {Coti Zelati, Michele and Drivas, Theodore D. and Gvalani,
              Rishabh S.},
     TITLE = {Mixing by statistically self-similar {G}aussian random fields},
   JOURNAL = {J. Stat. Phys.},
  FJOURNAL = {Journal of Statistical Physics},
    VOLUME = {191},
      YEAR = {2024},
    NUMBER = {5},
     PAGES = {Paper No. 61, 11},
       DOI = {10.1007/s10955-024-03277-w},
       URL = {https://doi-org.univaq.idm.oclc.org/10.1007/s10955-024-03277-w},
}
\bibliographystyle{plain}

\end{document}